\documentclass[12pt,a4paper,reqno]{amsart}
\usepackage{a4,verbatim}
\usepackage{hyperref,tikz}

\title[The geometric Toda equations.]{The geometric Toda equations for noncompact symmetric spaces.}

\author{Ian McIntosh}
\thanks{For the purposes of open access, the author has applied a Creative Commons Attribution  (CC BY) to any Author Accepted
Manuscript version arising from this submission.}
\address{Department of Mathematics\\ University of York\\ 
York YO10 5DD, UK}
\email{ian.mcintosh@york.ac.uk}
\date{November 18, 2024}

\newcommand{\Z}{\mathbb{Z}}

\newcommand{\C}{\mathbb{C}}
\newcommand{\uC}{\underline{\C}}

\newcommand{\R}{\mathbb{R}}

\newcommand{\RP}{\mathbb{RP}}

\newcommand{\CH}{\mathbb{CH}}

\newcommand{\Ct}{\C^\times}
\newcommand{\e}{\varepsilon}
\newcommand{\be}{\mathbf{e}}

\newcommand{\caD}{\mathcal{D}}
\newcommand{\caE}{\mathcal{E}}

\newcommand{\caG}{\mathcal{G}}

\newcommand{\fso}{\mathfrak{so}}
\newcommand{\fsl}{\mathfrak{sl}}
\newcommand{\fsp}{\mathfrak{sp}}
\newcommand{\fsu}{\mathfrak{su}}
\newcommand{\fgl}{\mathfrak{gl}}
\newcommand{\fa}{\mathfrak{a}}
\newcommand{\fb}{\mathfrak{b}}
\newcommand{\fc}{\mathfrak{c}}
\newcommand{\fd}{\mathfrak{d}}
\newcommand{\fe}{\mathfrak{e}}
\newcommand{\ff}{\mathfrak{f}}
\newcommand{\fg}{\mathfrak{g}}
\newcommand{\fh}{\mathfrak{h}}
\newcommand{\fk}{\mathfrak{k}}
\newcommand{\fm}{\mathfrak{m}}

\newcommand{\fp}{\mathfrak{p}}

\newcommand{\fs}{\mathfrak{s}}
\newcommand{\ft}{\mathfrak{t}}
\newcommand{\fu}{\mathfrak{u}}
\newcommand{\fv}{\mathfrak{v}}
\newcommand{\fz}{\mathfrak{z}}

\newcommand{\Aut}{\operatorname{Aut}}
\newcommand{\End}{\operatorname{End}}
\newcommand{\Hom}{\operatorname{Hom}}

\newcommand{\Ad}{\operatorname{Ad}}
\newcommand{\ad}{\operatorname{ad}}
\newcommand{\INT}{\operatorname{int}}
\newcommand{\Tr}{\operatorname{Tr}}

\newcommand{\<}{\langle}
\renewcommand{\>}{\rangle}
\newcommand{\Span}{\operatorname{Span}}

\newcommand{\rank}{\operatorname{rank}}

\renewcommand{\implies}{\Rightarrow}

\newcommand{\inn}{\mathit{inn}}

\newtheorem{thm}{Theorem}[section]
\newtheorem{prop}[thm]{Proposition}
\newtheorem{lem}[thm]{Lemma}
\newtheorem{cor}[thm]{Corollary}
\newtheorem{defn}[thm]{Definition}
\theoremstyle{remark}
\newtheorem{exam}[thm]{Example}
\newtheorem{rem}[thm]{Remark}

\numberwithin{equation}{section}

\begin{document}

\begin{abstract}
This paper has two purposes. The first is to classify all those versions of the Toda equations which govern the existence of 
$\tau$-primitive harmonic maps 
from a surface into a homogeneous space $G/T$ for which $G$ is a noncomplex noncompact simple real Lie group,
$\tau$ is the Coxeter automorphism which Drinfel'd \& Sokolov assigned to each affine Dynkin diagram, and $T$ is the
compact torus fixed pointwise by $\tau$. 
Here $\tau$ may be either an inner or an outer automorphism. 
We interpret the Toda equations over a compact Riemann surface $\Sigma$ as equations for a metric on a
holomorphic principal $T^\C$-bundle $Q^\C$ over  $\Sigma$ whose Chern connection, when combined with 
holomorphic field $\varphi$, produces a $G$-connection which is flat precisely when the Toda equations hold.
The second purpose is to establish when stability criteria for the pair $(Q^\C,\varphi)$ can be used to prove 
the existence of solutions.  We classify those real forms of the Toda equations for which this
pair is a principal pair and we call these \emph{totally noncompact} Toda pairs: stability theory then gives
algebraic conditions for the existence of solutions.  Every solution to the geometric Toda equations 
has a corresponding $G$-Higgs bundle. We explain how to construct this $G$-Higgs bundle directly from
the Toda pair and show that Baraglia's cyclic Higgs bundles arise from a very special case of totally noncompact cyclic Toda pairs.
\end{abstract}

\maketitle
%\tableofcontents

\section{Introduction.}

The version of the Toda equations we are concerned with here comes from the study of a special class of
harmonic maps from a Riemann surface $\Sigma$ into a Riemannian symmetric space $N$. In the case where surface is a torus and the
symmetric space is of compact type with simple compact isometry group $G$ the Toda equations have the form
\begin{equation}\label{eq:Todacompact}
\Delta_\Sigma w_j = \sum_{k=0}^r \hat C_{jk} e^{w_k},\quad\text{with}\ \sum_{j=0}^r m_jw_j=0.
\end{equation}
Here $w_j:\Sigma\to \R$ are smooth functions, $\Delta_\Sigma = d^*d$ is the (non-negative) Laplacian on functions on $\Sigma$, the matrix 
$\hat C$ is an $(r+1)\times(r+1)$ affine Cartan matrix associated to the complexification $\fg^\C$ of the Lie algebra of $G$, and the
positive integers $m_j$ are the coefficients of linear dependence of the rows of $\hat C$ normalized so that $m_0=1$  
(a summary of the relevant facts about affine Cartan matrices is given in \S \ref{sec:Coxeter} below).
Equations \eqref{eq:Todacompact} are a real form of what  is often called the affine Toda (field) equations: there is also a 
version when $\hat C$ is replaced by the standard Cartan matrix $C$ for $\fg^\C$, which we will refer to as the non-affine Toda equations. 

Equations \eqref{eq:Todacompact} govern the existence of a ``$\tau$-primitive'' harmonic map into a 
homogeneous space $G/T$ which lies over the symmetric space $N\simeq G/H$ by homogeneous projection. Here $T$ is a maximal
torus in the isotropy subgroup $H< G$ and $r=\dim(T)$. This $\tau$-primitive 
map projects down to produce a conformal harmonic, and hence minimal, surface in $N$. Here $\tau$ denotes a
specific finite order automorphism on $G$, called the Coxeter automorphism, whose fixed-point subgroup is $T$. 
The harmonic map into $G/T$ is often referred to as the $\tau$-primitive lift of the harmonic map to
$N$ and the idea has its origins in Riemannian twistor theory (see especially \cite{BurR,Raw,Bla}). 

In the literature the affine Cartan matrix $\hat C$ is usually taken to correspond to an extended Dynkin diagram, in which case $\tau$ is an
inner automorphism, $T< G$ is a maximal torus and $N$ is an inner symmetric space (one can also use a symmetry trick
to adapt these inner Coxeter Toda systems to get harmonic maps into outer symmetric spaces: see e.g.\
\cite{OliT,BolW,BolPW,Bar}).
However, to treat inner and outer symmetric spaces on equal footing we allow $\hat C$ to be the affine Cartan matrix of any affine
diagram, following Drinfel'd \& Sokolov \cite{DriS}.

To adapt the Toda equations to a noncompact real group $G$ it must be preserved by the Coxeter 
automorphism $\tau$ and the compact torus $T<G$ fixed pointwise. The effect on \eqref{eq:Todacompact} is to produce a 
different distribution of plus and minus signs
on the coefficients.  Classifying which of these are possible amounts to answering the question: which real forms $\fg$ of 
a complex simple Lie algebra $\fg^\C$ are invariant under a Coxeter automorphism $\tau$ which fixes a compact
toral subalgebra $\ft\subset\fg$?
We originally thought this question had been answered, at least for inner Coxeter automorphisms, by Carberry \&
Turner in \cite{CarT}, but closer inspection shows that they have overlooked many cases (essentially they overlook the
fact that a real involution can also change the sign as it swaps root vectors and is therefore not determined by its
action on roots alone).

Our classification goes as follows.  Let $\fg^\C$ have Dynkin diagram $\Gamma$. 
Recall (from e.g.\ \cite[Ch.~X]{Hel}) that for each symmetry
$\nu$ of $\Gamma$ there is a unique affine diagram $\Gamma^{(n)}$ where $n\in\{1,2,3\}$ is the order of $\nu$ 
(these can be found in Appendix \ref{app:diagrams}). To each
$\Gamma^{(n)}$ we assign a unique, up to conjugacy, Coxeter automorphism $\tau$ which is an inner automorphism when $n=1$
(the more usual case) and outer otherwise.
Every noncompact real form can be classified up to equivalence by its Cartan involution $\sigma$, thought of as a complex involution.
Since we are assuming $\fg^\C$ is simple the noncompact symmetric spaces we describe are those of Type III in the standard
classification (cf.\ \cite{Hel}), i.e., the real form $\fg$ is not itself a complex Lie algebra.
We say $\sigma$ is compatible with $\tau$ when it commutes with $\tau$ and fixes pointwise the fixed-point subalgebra of $\tau$. 
We say two such involutions are $\tau$-equivalent when they are conjugate in $\Aut(\fg^\C)$ by an automorphism which commutes
with $\tau$. The following theorem combines Prop.\ \ref{prop:classification} and Cor.\ \ref{cor:compatible} below.
\begin{thm}\label{thm:noncompact}
Let $\Gamma^{(n)}$, $n\in\{1,2,3\}$, be an affine Dynkin diagram with vertex set indexed by $I=\{0,\ldots,r\}$ and 
Coxeter automorphism $\tau$.  Every involution $\sigma$
of $\fg^\C$ which is compatible with $\tau$ is determined by a map $\ell:I\to \{-1,1\}$ for which $-1\in\ell(I)$
and satisfying $\prod_{j\in I} \ell(j)^{nm_j}=1$.  
Two such involutions are $\tau$-equivalent if and only if they are related by a symmetry of $\Gamma^{(n)}$. 
Every noncompact symmetric space of Type III has a least one compatible Coxeter automorphism.  
\end{thm}
The vertices of $\Gamma^{(n)}$ correspond to the simple affine roots, and the labels $\ell(j)$ are the eigenvalues of
$\sigma$ on the corresponding root space. Since those root spaces generate $\fg^\C$ under Lie bracket
this determines $\sigma$ completely, with the relation being the sole condition required to ensure that $\sigma$ is well-defined
as an automorphism of $\fg^\C$. The distribution
of signs within the Toda equations follows easily. Each labelling $\ell$ partitions $I$ into $I_+\cup I_-$ according to the
sign, and the corresponding affine Toda equations become
\begin{equation}\label{eq:Todanoncompact}
\Delta_\Sigma w_j = \sum_{k\in I_+}^r \hat C_{jk} e^{w_k} - \sum_{k\in I_-}^r \hat C_{jk} e^{w_k},\quad\ \sum_{j=0}^r m_jw_j=0.
\end{equation}

When $\Sigma$ is a torus these equations are the zero curvature conditions for a certain flat $G$-connection, where $G$ is
the real form determined by the choice of signs, and their solution
corresponds to a $\tau$-primitive map $\psi:\tilde\Sigma\to G/T$ which is equivariant with respect to the holonomy
representation $\pi_1\Sigma\to G$ of that flat connection.

When the compact Riemann surface $\Sigma$ is not a torus these equations need to be modified to fulfil the same purpose.
In the second half of the paper we derive these \emph{geometric Toda equations}. Motivated
by the relationship between Higgs bundles and harmonic maps, we take the point of view that the Toda equations are
really the equations for a metric, or equivalently a reduction of structure group, from a holomorphic principal $T^\C$-bundle $Q^\C$
over $\Sigma$ to a $T$-bundle $Q\subset Q^\C$. The equations governing the metric require an initial choice of Hermitian metric on
$\Sigma$ and the bundle $Q^\C$, together with a holomorphic section $\varphi$ of the associated bundle $Q^\C(\fg_1)\otimes
K_\Sigma$, where $\fg_1\subset\fg^\C$ is the direct sum of root spaces for the simple affine roots. 

Since $T^\C$ is abelian and its action on $\fg_1$ is completely reducible we can equate this pair $(Q^\C,\varphi)$
with a choice of line bundles
and holomorphic sections.  These choices determine the precise form of our geometric Toda equations.
\begin{defn}\label{defn:geomToda}
Fix a compact Riemann surface $\Sigma$ of genus $g$, equipped with a K\"{a}hler metric of constant curvature $2-2g$ and
let $\omega_\Sigma$ denote its K\"{a}hler form.
Fix an affine Dynkin diagram, with Cartan matrix $\hat C$, vertex index set $I=\{0,1,\ldots,r\}$ and simple affine roots
$\{\bar\alpha_k:k\in I\}$. Choose a vertex
labelling $\ell:I\to \{-1,1\}$ satisfying the conditions of Theorem \ref{thm:noncompact} and let
$I=I_+\cup I_-$ be the partition it induces.  Choose holomorphic 
line bundles $Q_k(\C)$, $k\in I$, over $\Sigma$ such that $\otimes_{k\in I} Q_k(\C)^{m_k}\simeq\uC$ and set $d_k=\deg(Q_k)$.  
For each $k\in I$ take the unique, up to constant scaling, Hermitian metric on $Q_k(\C)$ 
whose Chern connection
has curvature $-id_k\omega_\Sigma$. Finally, for each $k\in I$ choose a holomorphic section $\varphi_k\in H^0(Q_k(\C)\otimes K_\Sigma)$. 
Then the \emph{geometric Toda equations} for this data are the system of elliptic p.d.e.\
\begin{equation}\label{eq:geomToda}
\Delta_\Sigma w_j = \sum_{k\in I_+}\|\varphi_k\|^2\tfrac{|\alpha_k|^2}{2} e^{w_k}\hat C_{jk} - 
\sum_{k\in I_-}\|\varphi_k\|^2\tfrac{|\alpha_k|^2}{2} e^{w_k}\hat C_{jk} - d_j, \quad j\in I,
\end{equation}
for functions $w_j:\Sigma\to \R$ with $\sum_{j=0}^rm_jw_j=0$.
\end{defn}
We call the pair $(Q^\C,\varphi)$ a \emph{Toda pair}. We say it is  \emph{cyclic} when all $\varphi_k$ are non-trivial and
\emph{non-cyclic} otherwise. A particular case of the latter is the \emph{simple non-affine} case, for which $\varphi_k=0$
only for $k=0$, and we show that every non-cyclic case is a union of simple non-affine cases.
For the holomorphic section $\varphi_k$ to be non-trivial the degree $d_k$ must satisfy $d_k\geq 2-2g$.
Hence a necessary condition for the existence of cyclic pairs is that $g\geq 1$ and the vector $(d_0,\ldots,d_r)$ of degrees lies in
a bounded polytope
\begin{equation}\label{eq:intro_d_k}
2-2g\leq d_j,\ j=0,\ldots,r,\quad \sum_{k=0}^r m_kd_k=0.
\end{equation}
Note that for $g=1$ this polytope is a single point $(0,\ldots,0)$. Since $K_\Sigma\simeq\uC$ this forces $Q_k(\C)\simeq\uC$ to obtain non-trivial
$\varphi_k$ and leads to the more standard form of the Toda equations where all coefficients are constants.

Regarding existence of solutions, when $G$ is noncompact the only general method we are aware of is when we can treat
$(Q^\C,\varphi)$ as a \emph{principal pair} in the language of Bradlow et.\ al \cite{BraGM} (cf.\ Banfield \cite{Ban}).
We show that for cyclic Toda pairs this only applies when $\ell(j)=-1$ for $j=0,\ldots,r$. We call this case \emph{totally
noncompact} since it is the case where every affine simple root is noncompact (its root space lies in the
complexification of the noncompact summand $\fm$ in the Cartan decomposition $\fg =\fh\oplus\fm$).
There is precisely one such case for every affine Dynkin
diagram except $\fa_{2k}^{(1)}$, since this is the only case for which the relation $\prod_{j\in I}\ell(j)^{nm_j}=1$ cannot hold. 
It is not immediately obvious which symmetric spaces these are in the standard classification of, say, \cite[Ch.~X, Table V]{Hel} 
so we provide a list in Table \ref{table1}.

The geometric Toda equations for this totally noncompact case are
\begin{equation}\label{eq:TodaTotNon}
\Delta_\Sigma w_j + \sum_{k\in I}\|\varphi_k\|^2 \tfrac{|\alpha_k|^2}{2} e^{w_k}\hat C_{jk} +d_j 
= 0, \quad j\in I.
\end{equation}
In this case by checking the stability conditions for the Toda pair we prove the following (Prop.\ \ref{prop:cyclic} and
\ref{prop:non-affine}).
\begin{thm}\label{thm:exist}
Let $(Q^\C,\varphi)$ be a totally noncompact Toda pair. If it is cyclic 
then \eqref{eq:TodaTotNon} has a unique solution. If it is simple non-affine then a solution exists (and it is unique) whenever
the vector of degrees $(d_1,\ldots,d_r)$ lies in the convex polytope
\begin{equation}%\label{eq:Rdintro}
2-2g\leq d_j,\quad \sum_{k=1}^rR_{jk}d_k<0,
\end{equation}
where $R_{jk}$ are the entries to the inverse of the matrix whose entries are $\<\alpha_j,\alpha_k\>$.
In particular, this requires $g\geq 2$.
\end{thm}
For any Toda pair the existence of a solution to \ref{eq:geomToda} means there is an equivariant $\tau$-primitive harmonic
map $\psi:\tilde\Sigma\to G/T$ on the universal cover of $\Sigma$ (the equivariance is with respect to the holonomy
representation $\pi_1\Sigma\to G$ of the flat $G$-connection). It is a property of $\tau$-primitive maps that after homogeneous
projection $G/T\to G/H\simeq N$ we obtain an equivariant harmonic map $f:\tilde\Sigma\to N$ into the symmetric space. When
$G$ is noncompact this has a corresponding $G$-Higgs bundle. In \S \ref{sec:Higgs} we explain how the Toda pair
$(Q^\C,\varphi)$ is related to this Higgs bundle. We show that this Higgs bundle is a Hodge bundle (i.e., fixed by the
$\C^\times$-action on Higgs bundles) if and only if the Toda pair is non-cyclic. 

Section \ref{sec:Higgs} ends with an explanation of how Baraglia's
cyclic Higgs bundles\footnote{These were originally studied many years earlier by Aldrovandi \& Falqui \cite{AldF} although they were not called
``cyclic''. We call these ``Baraglia's'' to distinguish them from Collier's \cite{ColThesis} use of ``cyclic Higgs bundle'' which 
actually means cyclotomic in the sense of Simpson \cite{Sim09}. See also Remark \ref{rem:cyclotomic}.}
are the totally noncompact Toda pairs for an inner Coxeter automorphism in the
special case that $Q_j(\C)\simeq K_\Sigma^{-1}$, and hence $\varphi_j$ is a nonzero constant for $j=1,\ldots,r$, 
to obtain the system
\begin{equation}\label{eq:KToda}
	\Delta_\Sigma w_j + \sum_{k=1}^r \hat C_{jk}c_ke^{w_k} + \hat C_{j0}\tfrac{|\alpha_0|^2}{2}\|\varphi_0\|^2 e^{w_0} +d_j 
= 0, \quad j\in I,
\end{equation}
for $\varphi_0\in H^0(K_\Sigma^{m/n})$ and certain constants $c_k$.  These choices can be made 
even when $\tau$ is an outer Coxeter automorphism, providing  \emph{two} distinct systems of equations
for the Lie algebras $\fa_{2k-1}$, $\fd_r$ and $\fe_6$ (and three distinct versions for $\fd_4$). 

Existence of solutions to \eqref{eq:KToda} has been a topic of significant recent interest. For Bargalia's cases existence
when $g\geq 2$ is guaranteed by nonabelian Hodge theory: Hitchin's Higgs bundles are all stable.  
Miyatake \cite{Miy} provides an alternative existence theorem via his study of generalized Kazdan-Warner equations.
In fact all totally noncompact systems \eqref{eq:TodaTotNon} fit into his theory, providing an alternative proof of
Theorem \ref{thm:exist}. 
Solutions over noncompact domains, especially $\Sigma\simeq \C^\times$ and assuming rotational symmetry, have been
treated in \cite{GueL,GueIL,Moc,LiM1,LiM2}, motivated by the relationship with the $tt^*$ equations and
quantum cohomology (see, for example, \cite{Gue}, for a discussion of the motivation). 
Most of the existence results treat the $\fa_r^{(1)}$ equations
with an additional symmetry (which for $r=2k$ gives the same system as for $\fa_{2k}^{(2)}$),
but we note that the method of sub/super-solutions elucidated in \cite[\S 5.1]{LiM1}, derived from the method used in \cite{GueL}, can 
be applied to any totally
noncompact Toda system \eqref{eq:TodaTotNon} since the affine Cartan matrix has the property that $\hat C_{jk}<0$
whenever $j\neq k$. 

Beyond these results lies the question of the existence of solutions to the general noncompact Toda system \eqref{eq:Todanoncompact}. To my
knowledge  this has not been studied outside the $\fa_2^{(2)}$ case (see \cite{LofMsurv} for a survey) where the equation is 
effectively scalar. 
A surprising feature of the totally noncompact Toda systems \eqref{eq:TodaTotNon} over a compact
domain is that the norms $\|\varphi_j\|$ play no role in the existence theory (apart from whether or not they are identically zero).
One does not expect this to be the case more generally: one expects to need bounds on $\|\varphi_k\|$ whenever $k\in I_+$,
to control the terms with ``bad'' (i.e., positive) signs in \eqref{eq:Todanoncompact}. 

\smallskip\noindent
\textbf{Acknowledgments.} This work arose out of discussions  with Martin Guest on the topic of real forms of the
Toda equations while the author was
visiting him at Waseda University, Japan, during July 2022 as a Japan Society for the 
Promotion of Science Short Term Fellow. Discussions of an early draft also took place during a visit to Waseda University
in April 2024.  The author is grateful to JSPS for their support 
and to Martin for the stimulating discussions and critical feedback.

\section{Primitive harmonic maps.}
\subsection{Some reductive homogeneous geometry.}\label{ss:homog}

For what follows we will need the following facts about reductive homogeneous spaces of noncompact Lie groups
$G$. We will deal exclusively with real noncompact simple Lie groups $G$ which are not complex Lie groups. 
We will always choose $G$ to have trivial centre (i.e., to be the adjoint group).
Given a maximal compact subgroup $H<G$ the homogeneous space
$G/H$ can be given the structure of a noncompact symmetric space (these are the symmetric spaces of Type III in Cartan's
classification cf.\ \cite[Ch 5,Thm 5.4]{Hel}). Let $\fh\subset \fg$ denote the Lie algebras of $H<G$. The real form
$\fg\subset\fg^\C$ is determined by a real involution $\rho$ on $\fg^\C$ whose fixed point subalgebra is $\fg$. 
Write the corresponding Cartan decomposition as $\fg=\fh+\fm$. Let $\fu$ be a compatible compact real form, i.e., whose
real involution $\kappa$ commutes with $\rho$. Then $\fh = \fg\cap \fu$ and $\fm=\fg\cap(i\fu)$.

Let $T< H$ be any closed subgroup. Then 
its  Lie subalgebra $\ft\subset\fg$ is part of a reductive decomposition $\fg =
\ft+\fp$: since $\fh$ the Killing form $\<\ ,\ \>$ is non-degenerate on $\fg$ we can choose $\fp =
\ft^\perp$ to be the orthogonal complement with respect to the Killing form. 
Then $[\ft,\fp]\subset\fp$ and $G/T$ is a reductive homogeneous space whose tangent space at any
point is modelled on $\fp$.  
We equip $G/T$ with the $G$-invariant Riemannian metric given by restricting the $\Ad_H$-invariant Hermitian form
\begin{equation}\label{eq:mu}
\mu:\fg^\C\times\fg^\C\to \C,\quad \mu(\xi,\eta) = -\<\xi,\kappa\eta\>,
\end{equation} 
to $\fp$. When $G/H$ is equipped with the symmetric space metric given by the
restriction of $\mu$ to $\fm$ the  homogeneous projection $\pi:G/T\to G/H$ is a Riemannian submersion.

Let $\omega_G:TG\to\fg$ be the left Maurer-Cartan form and use the reductive decomposition of $\fg$ to write this as 
$\omega_\ft+\omega_\fp$. The first factor, $\omega_\ft$, provides the \emph{canonical
connection} for the right principal $T$-bundle $G\to G/T$, while the second factor is tensorial and therefore descends 
to a $1$-form $\beta:T(G/T)\to [\fp]$, where $[\fp] = G\times_T\fp$ is the
associated bundle over $G/T$. In fact $\beta$ is an isomorphism known as the \emph{Maurer-Cartan
form of $G/T$}. Now consider $[\fp]$ as a subbundle of $[\fg]$. The latter has canonical trivialisation 
\[
G\times_T \fg\simeq G/T\times \fg;\quad (g,\xi)_T\mapsto (gT,\Ad g\cdot\xi).
\]
This carries two
connections: the flat one $d$ from this trivialisation and the connection $D$ induced from $\omega_\ft$. It is shown in
\cite[Prop 1.1]{BurR} that $d = D+\beta$ and therefore the curvatures satisfy
\[
0=F^d = F^D + D\beta +\tfrac12[\beta\wedge\beta].
\]
Now $[\fg]=[\ft]+[\fp]$ and this splitting is $D$-parallel. Since $[\ft,\fp]\subset \fp$ when we project this equation onto
the subbundles $[\ft],[\fp]$ we obtain two equations
\begin{eqnarray*}
F^D + \tfrac12[\beta\wedge\beta]_\ft &=&0, \notag\\
D\beta + \tfrac12[\beta\wedge\beta]_\fp &=&0.\notag 
\end{eqnarray*}
Now let $M$ be any smooth manifold and
$\psi:M\to G/T$ a smooth map. Let $Q=\psi^{-1}G$ be the pull-back $T$-bundle over $M$ and set $\theta = \psi^*\beta$. For any
$\Ad_T$-invariant subspace $\fv\subset \fg$ define $Q(\fv) = Q\times_T \fv = \psi^*[\fv]$. This inherits the
connection $\nabla = \psi^*D$. Then we obtain what can be considered the structure equations for maps into $G/T$:
\begin{eqnarray}
F^\nabla + \tfrac12[\theta\wedge\theta]_\ft &=&0, \label{eq:struct1}\\
\nabla\theta + \tfrac12[\theta\wedge\theta]_\fp &=&0.\label{eq:struct2}
\end{eqnarray}
These are exactly the conditions for the connection $\nabla +\theta$ on $Q(\fg)$ to be flat. By considering the corresponding
flat connection on the principal $G$-bundle $Q\times_TG$ one obtains the following well-known theorem.
\begin{thm}\label{thm:struct}
Let $Q\to M$ be a principal $T$-bundle with connection $\nabla$ and $\theta\in\caE^1(Q(\fp))$ be such that
\eqref{eq:struct1} and \eqref{eq:struct2} both hold. Then there exists a smooth map $\psi:\tilde M\to G/T$ 
from the universal cover of $M$, unique up to isometries,
for which $\psi^*\beta =\theta$ and $\psi$ is equivariant with respect to a representation $\chi:\pi_1 M\to G$.
\end{thm}
The representation $\chi$ is the holonomy of the flat connection $\nabla+\theta$.

The application which interests us most is when $M$ is a Riemann surface $\Sigma$ 
and $\theta$ satisfies the conditions
\begin{equation}\label{eq:theta}
[\theta\wedge\theta]_\fp=0\ \text{and}\ \nabla''\theta'=0,
\end{equation}
where $T^\C\Sigma  = T'\Sigma\oplus T''\Sigma$ is the type decomposition determined by the complex structure on $\Sigma$.
Both of these conditions hold when $\psi$ is a $\tau$-primitive harmonic map, the definition of which 
we will now recall.

\subsection{$\tau$-primitive harmonic maps.}

Suppose we have an automorphism $\tau$ of 
$\fg^\C$ of finite order $m\geq 3$ which commutes with
$\rho$ and whose fixed points all lie in $\fh^\C$. It induces a $\Z_m$-grading on the complexified Lie algebra $\fg^\C$ 
which we will write as
\begin{equation}\label{eq:grading}
\fg^\C = \bigoplus_{k\in\Z_m} \fg^\tau_k, \quad [\fg^\tau_j,\fg^\tau_k]\subset \fg^\tau_{j+k}.
\end{equation}
Here $\fg^\tau_k$ is the eigenspace of $\tau$ for eigenvalue $e^{2\pi ik/m}$. Clearly $\rho$ maps $\fg^\tau_k$ to
$\fg^\tau_{-k}$. We define  $\ft=\fh\cap\fg_0^\tau$
to be the real Lie subalgebra $\ft\subset \fg$ with $\ft^\C=\fg^\tau_0$. Because $\tau$ has finite order $\fg^\tau_0$ is
a reductive subalgebra with complement $\oplus_{k\neq 0}\fg^\tau_k$. Since the latter is also $\tau$-invariant it is
the complexification of a complementary subspace $\fp\subset\fg$, and $\fg=\ft+\fp$ is a reductive decomposition.
Moreover, $\<\fg^\tau_j,\fg^\tau_k\>=0$ when $j+k\neq 0\in\Z_m$ \cite[Ch.\ X, Lemma 5.1]{Hel}, so $\fp=\ft^\perp$ with respect to the Killing form.
As above $T< H$ denotes the Lie subgroup with Lie algebra $\ft$ and $G/T$ is a reductive homogeneous space. 

The adjoint action of $T$ on each $\fg^\tau_k$ provides a bundle decomposition
\[
T^\C(G/T)\simeq [\fp^\C] = \bigoplus_{k\neq 0} [\fg^\tau_k],
\]
and the subbundle $[\fg^\tau_1]$ is called the \emph{primitive distribution} with respect to $\tau$.
\begin{defn}
A smooth map $\psi:\tilde\Sigma\to G/T$ of a Riemann surface $\Sigma$ is $\tau$-primitive when $\partial\psi:T'\Sigma\to T^\C(G/T)$
has its image in $[\fg^\tau_1]$.
\end{defn}
The crucial property here is that $[\fg^\tau_1,\rho(\fg^\tau_1)]=[\fg^\tau_1,\fg^\tau_{-1}]\subset \ft^\C$. This makes
$\tau$-primitive maps $F$-holomorphic with respect to the $G$-invariant horizontal $F$-structure on $G/T$ characterised by
the $\Ad_T$-invariant $F\in\End(\fp^\C)$ defined by
\begin{equation}\label{eq:Fstructure}
F = \begin{cases} 
\ i\ \text{on}\ \fg^\tau_1,\\
-i\ \text{on}\ \fg^\tau_{m-1},\\
\ 0\ \text{on}\ \fg^\tau_k\ \text{for}\ k\neq 1,m-1.
\end{cases}
\end{equation}
$F$-structures of this type were studied by Black \cite{Bla}. The following theorem is a special case of his results.
\begin{thm}[\cite{Bla}]\label{thm:Black}
A smooth map $\psi:\tilde\Sigma\to G/T$ of a Riemann surface which is $\tau$-primitive is harmonic. Further, 
let $\pi:G/T\to G/H$ be the homogeneous projection, 
then $f=\pi\circ\psi:\tilde\Sigma\to G/H$ is also harmonic.
\end{thm}
Our notation anticipates our application, in which $\tilde\Sigma$ will be the universal cover of a compact Riemann surface
$\Sigma$. 
\begin{rem}
In \cite[Lemma 5.4]{Bla}  a more general result is given, which asserts that both maps are equi-harmonic
(harmonic with respect to every $G$-invariant metric). But this level of generality requires the assumption that 
the $T^\C$ action on $\fp^\C$ decomposes
it into distinct irreducible summands. This assumption will not hold when we take $\tau$ to be an outer Coxeter
automorphism, since then $T^\C$ is not maximal abelian in $G^\C$ (see below). Nonetheless for our particular choice of
metric and $F$-structure Black's arguments apply. An alternative proof which applies to our situation, using local frames, 
can be found in \cite[Thm 3.7]{BurP}.
\end{rem}
In fact the map $f$ is weakly conformal harmonic, i.e., a minimal surface. 
\begin{prop}\label{prop:conformal}
The map $f=\pi\circ\psi$ obtained from Theorem \ref{thm:Black} is weakly conformal, i.e., $(f^*\mu)^{(2,0)}=0$ where $\mu$
denotes the metric on $G/H$.
\end{prop}
\begin{proof}
Since $f^*\mu = \psi^*\pi^*\mu$ we have
\[
f^*\mu^{(2,0)} = \<P_\fm\partial\psi,P_\fm\partial\psi\>,
\]
where $P_\fm:[\fp^\C]\to[\fm^\C]$ is the projection derived from the splitting $\fp^\C = (\fh^\C\cap\fp^\C)\oplus\fm^\C$.
Since $\psi$ is $\tau$-primitive $P_\fm\partial\psi$ takes values in $[\fg^\tau_1]$ and now $\<\fg^\tau_1,\fg^\tau_1\>=0$ 
implies $(f^*\mu)^{(2,0)}=0$.
\end{proof}
Now we can apply the theory of subsection \ref{ss:homog} to observe that a $\tau$-primitive map satisfies the equations
\begin{eqnarray}
 F^\nabla +[\varphi\wedge \rho\varphi]&=&0,\label{eq:F1}\\
\nabla''\varphi&=&0,\label{eq:F2}
\end{eqnarray}
where $\varphi = \partial\psi$, using the fact that $\rho\varphi = \bar\partial\psi$. Here we are using $\rho$ to also denote the real
involution on $Q(\fp^\C)$ fixing $Q(\fp)$. The second equation holds because 
$\theta = \varphi+\rho\varphi$ decomposes $\theta$ into $[\fg^\tau_1]$ and $[\fg^\tau_{-1}]$ components, which are independent and
$\nabla$-invariant. In particular, we have the following corollary of Theorem \ref{thm:struct}. 
\begin{cor}\label{cor:Q}
Let $Q\to \Sigma$ be a principal $T$-bundle over a Riemann surface $\Sigma$, with connection $\nabla$. Let
$K_\Sigma$ be the canonical bundle of $\Sigma$ and let $\varphi$ be a holomorphic section of $Q(\fg^\tau_1)\otimes K_\Sigma$ 
where the holomorphic structure is given by $\nabla''$. If $(\nabla,\varphi)$ satisfy
\eqref{eq:F1} and \eqref{eq:F2} then there exists a $\pi_1\Sigma$-equivariant $\tau$-primitive map $\psi:\tilde \Sigma\to G/T$ for which
$\partial\psi = \varphi$. It is unique up to isometries.
\end{cor}
Our particular application of Theorem \ref{thm:Black} and Corollary \ref{cor:Q} will be when $\tau$ is the Coxeter automorphism and $T$ its fixed-point subgroup. 
Then the Toda equations are the conditions that \eqref{eq:F1} and \eqref{eq:F2} hold.

\section{Coxeter automorphisms and noncompact Lie algebras.}\label{sec:Coxeter}

We now want to apply the results above to the situation where $\tau$ is what we will call the Coxeter automorphism for an
affine Dynkin diagram. This requires our noncompact real form be compatible $\tau$. To enable the classification of these, 
we fix first a compact real
form $\fu\subset\fg^\C$, with real involution $\kappa$. All other real forms are identified, up to conjugacy, by their Cartan
involution $\sigma\in\End(\fg^\C)$, for which $\rho=\sigma\kappa=\kappa\sigma$.
We will write $\fh^\C=\fg_0^\sigma$, $\fm^\C=\fg_1^\sigma$ when it becomes necessary to compare Cartan involutions.
Our strategy will be to describe the
Coxeter automorphism $\tau$ on $\fg^\C$ first and then explain which complex involutions $\sigma$ are compatible with it to provide the 
$\tau$-primitive distribution over $G/T$ such that $T< H$.

\subsection{The Coxeter automorphism for an affine diagram.}

To describe the Coxeter automorphism we need to summarise the relevant parts of the theory of affine root systems from 
\cite[Ch.~X, \S 5]{Hel}.
Given $\fg^\C$ as above let $\nu$ be an automorphism of the Dynkin diagram of $\fg^\C$. We include the identity
automorphism, so that $\nu$ has order $n\in\{1,2,3\}$. By fixing canonical generators we can construct an automorphism of $\fg^\C$
of order $n$, which we also call $\nu$, which represents this symmetry (see Appendix \ref{app:roots}).  
The eigenspaces of $\nu$ equip $\fg^\C$ with a $\Z_n$-grading which we will write as
\[
\fg^\C = \bigoplus_{j\in\Z_n} \fg^\nu_j,\quad \fg_j^{\nu} = \{\xi\in\fg^\C: \nu(\xi) = e^{2\pi ij/n}\xi\}.
\]
It can be shown that $\fg_0^\nu$ is a simple Lie subalgebra of $\fg^\C$.
Let $\ft^\C$ be a choice of Cartan subalgebra for $\fg^\nu_0$.
A root for $\fg^\C$ relative to $(\nu,\ft^\C)$ is a
pair $\bar\alpha=(\alpha,j)\in (\ft^\C)^* \times \Z_n$ for which
\[
\fg^{\bar\alpha} = \{\xi\in\fg^\nu_j:[h,\xi]=\alpha(h)\xi\ \forall\ h\in\ft^\C\}\neq \{0\}.
\]
Let $\bar R$ be the set of all these roots and $\bar R^0$ be the subset of those for which $\alpha=0$. 
Then $\fg^{\bar\alpha}$ is one-dimensional for $\bar\alpha\in \bar R-\bar R^0$, whereas  
$\ft^\C = \fg^{(0,0)}$ and the centralizer $\fz(\ft^\C)$ of $\ft^\C$ in $\fg^\C$ is a Cartan subalgebra of $\fg^\C$ with
\[
\fz(\ft^\C) = \bigoplus_{\bar\alpha\in\bar R^0}  \fg^{\bar\alpha}.
\]
Then 
\[
\fg^\C = \fz(\ft^\C) \oplus \left( \bigoplus_{\bar\alpha\in\bar R-\bar R^0} \fg^{\bar\alpha}\right).
\]

The Killing form is
non-degenerate on $\ft^\C$ and hence for every $\bar\alpha\in \bar R-\bar R^0$ there is a unique $h_\alpha\in\ft^\C$ for which
$\alpha = \< h_\alpha,\ \>$ and one defines
\[
\<\alpha,\beta\> = \<h_\alpha,h_\beta\>.
\]
For such roots $[\fg^{\bar\alpha},\fg^{-\bar\alpha}]=\C.h_\alpha$.
The following theorem summarises the results we need regarding these roots (see, e.g.\ Helgason \cite[Ch.~X,\S 5]{Hel}).
\begin{thm}\label{thm:affine}
Let $(\fg^\C,\nu,\ft^\C)$ be as above.
\begin{enumerate}
\item If $\alpha_1,\ldots,\alpha_r\in(\ft^\C)^*$ is a basis of simple roots for 
$\fg_0^\nu$ then $\bar R$ is generated by roots $\bar\alpha_0,\bar\alpha_1,\ldots,\bar\alpha_r$. Here  $\bar\alpha_j =
(\alpha_j,0)$ for $j\neq 0$ and 
\[
\bar\alpha_0=
\begin{cases}(\alpha_0,0)\ \text{when}\ n=1,\\
(\alpha_0,1)\ \text{when}\ n\neq 1, 
\end{cases}
\]
is characterised by the property that 
$\bar\alpha_0-\bar\alpha_j\not\in\bar R$ for $j\neq 0$. 
There are positive integers $m_0=1,m_1,\ldots,m_r$ such that $\sum_{j=0}^rm_j\alpha_j=0$, equally,
\[
\sum_{j=0}^r nm_j\bar\alpha_j = (0,0)\in (\ft^\C)^*\times \Z_n.
\]
\item The affine Cartan matrix for $(\fg^\C,\nu)$ is defined by
\[
\hat C_{jk} = 2\frac{\<\alpha_j,\alpha_k\>}{\<\alpha_k,\alpha_k\>},\quad 0\leq j,k\leq r.
\]
By the previous part it has corank $1$.
\end{enumerate}
\end{thm}
We will call $\bar B=\{\bar\alpha_0,\ldots,\bar\alpha_r\}$ a system of \emph{simple affine roots\footnote{This is a slight
abuse of usage, as the affine roots are usually the elements of $(\ft^\C)^*\times\Z$ which reduce to ours under
$\Z\mapsto \Z_n$.} for $\bar R$}.   
The corresponding affine diagram is said to be \emph{of type $n$}.
We will call the integers $m_0,\ldots,m_r$ the \emph{affine coefficients} of the affine diagram corresponding to $\hat C$.
When $\nu$ is the identity, and hence $\fg^\C=\fg_0^\nu$, $\alpha_1,\ldots,\alpha_r$ will be a basis of simple roots for
$\fg^\C$ and $\alpha_0$ will be the lowest root with respect to these.

Set 
\begin{equation}\label{eq:hj}
h_j=\frac{2}{\<\alpha_j,\alpha_j\>}h_{\alpha_j},\quad j=0,1,\ldots,r,
\end{equation}
so that
\[
\alpha_j(h_k) =\hat C_{jk}.
\]
Then we can choose elements $e_j\in\fg^{\bar\alpha_j}$, $f_j\in\fg^{-\bar\alpha_j}$ such that
we have the relations
\begin{equation}\label{eq:Weyl}
[e_j,f_k]=-\delta_{jk} h_{j},\quad [h_j,e_k]= \hat C_{kj}e_k,\quad [h_j,f_k]= -\hat C_{kj}f_k.
\end{equation}
This collection $\{h_j,e_j,f_j:j=0,\ldots,r\}$ is a system of generators for $\fg^\C$. 
Define $m=\sum_{j=0}^r nm_j$ and call this the \emph{Coxeter number} for $(\fg^\C,\nu)$. 
By assigning each generator $e_j$ the weight $1\in\Z_m$ and each $f_j$ weight $-1\in\Z_m$ one obtains 
a $\Z_m$-grading on $\fg^\C$. We will write this grading as
\[
\fg^\C = \bigoplus_{k\in\Z_m} \fg^\tau_k 
\]
and notice that 
\begin{equation}
\fg^\tau_0=\ft^\C,\quad \fg^\tau_1 = \Span_\C\{e_0,\ldots,e_r\},\quad \fg^\tau_{-1} = \Span_\C\{f_0,\ldots,f_r\}.
\end{equation}
\begin{defn}\label{defn:Cox}
The Coxeter automorphism for $(\fg^\C,\nu,\ft^\C,\bar B)$ is the order $m$ automorphism $\tau$ which has $\fg^\tau_k$ as its eigenspace 
of eigenvalue $e^{2\pi i k/m}$.
\end{defn}
In the language of Kac \cite{Kac} (cf.\ \cite[Ch.\ X]{Hel}) our Coxeter automorphism is the unique, up to conjugacy,
automorphism of type $(1,\ldots,1;n)$ on $\fg^\C$. In particular, $\tau$ is an inner automorphism when $n=1$ and 
an outer automorphism otherwise. Clearly there is exactly one Coxeter automorphism up to conjugacy for 
every affine Dynkin diagram (equally, for every affine Cartan matrix).

The next lemma gives an explicit expression for $\tau$ given $(\fg^\C,\nu,\ft^\C,\bar B)$. 
\begin{lem}\label{lem:Cox}
Let $x\in\ft^\C$ be the unique element for which $\alpha_j(x)=1$ for $j=1,\ldots,r$. Then the Coxeter automorphism for
$(\fg^\C,\nu,\ft^\C,\bar B)$ is given by 
\begin{equation}\label{eq:Cox}
\tau = \nu\exp(\tfrac{2\pi i}{m} \ad x) = \exp(\tfrac{2\pi i}{m} \ad x)\nu.
\end{equation}
\end{lem}
\begin{proof}
Clearly the right hand side of \eqref{eq:Cox} fixes $\ft^\C$ pointwise so it suffices to show 
that $\fg^\tau_1$ is an eigenspaces of the correct eigenvalue. 

By definition of $e_j$, $\nu(e_j)=e_j$ for $j=1,\ldots,r$ and $\nu(e_0) = \exp^{2\pi i/n}e_0$.
So for $j=1,\ldots,r$, since $[x,e_j]=\alpha_j(x)e_j=e_j$ 
\[
\nu\exp(\tfrac{2\pi i}{m} \ad x)(e_j) = \nu(e^{2\pi i/m}e_j) = e^{2\pi i/m}e_j.
\]
For $j=0$ we have $\alpha_0=-\sum_{j=1}^r m_j\alpha_j$ so 
\[
[x,e_0] = -\sum_{j=1}^r m_j e_0 = (1-\tfrac{m}{n})e_0.
\]
Therefore
\[
\nu\exp(\tfrac{2\pi i}{m} \ad x)(e_0) = \nu(e^{2\pi i(1-m/n)/m)}e_0) = e^{2\pi i/m}e^{-2\pi i/n}\nu(e_0) =e^{2\pi
i/m}e_0.
\]
\end{proof}
\begin{rem}\label{rem:innerouter}
The only Lie algebras whose Dynkin diagrams admit a non-trivial symmetry are $\fa_r$ ($r\geq 2$), $\fd_r $ ($r\geq 3$) and $\fe_6$. 
These will have one inner Coxeter automorphism, call it $\tau_\inn$, and at least one outer Coxeter
automorphism (with $\fd_4$ being the only case with two distinct outer Coxeter automorphisms). However, it is only for 
$\fa_{2k}$ that the outer Coxeter automorphism equals
$\nu\tau_\inn$. Indeed, this is the only case where the
eigenspace $\fg^\tau_1$ for the outer Coxeter automorphism is contained in that for $\tau_\inn$. This can be seen clearly
in Appendix \ref{app:roots} where  the root space $\fg^{\bar\alpha_0}$ is identified for each case.
\end{rem}

Finally, we want to understand the properties of elements of $\fg_1^\tau$. 
Consider the adjoint action of $T^\C$ on $\fg_1^\tau$. Write $\xi=\sum_{j=0}^r c_je_j$ for $c_j\in\C$. 
Then for every $h\in\ft^\C$
\[
\exp(\ad h)\cdot \xi = \sum_{j=0}^r c_je^{\alpha_j(h)} e_j.
\]
It follows that when $c_j\neq 0$ for all $j$ this orbit is the zero level set of the polynomial
\[
P(X_0,\ldots,X_r)=\prod_{j=0}^r X_j^{m_j}- \prod_{j=1}^r c_j^{m_j} ,
\]
where $X_j:\fg_1^\tau\to \C$ are the coordinates for which $X_j(\xi)=c_j$. Thus the orbit is Zariski closed when $c_j\neq
0$ for all $j$. If $c_j=0$ for any $j$ then  then the Zariski closure of the orbit contains $0$. By results 
from Vinberg \cite{Vin}, an element of $\fg_1^\tau$ is semisimple if and only if its $T^\C$-orbit is Zariski closed 
and nilpotent if and only if the Zariski closure of its orbit contains $0$. Thus we deduce the following.
\begin{lem}
An element $\sum_{j=0}^r c_j e_j \in \fg_1^\tau$ is semisimple if and only if $c_j\neq 0$ for all $j$ and nilpotent
otherwise. 
\end{lem}
We will call elements of $\fg_1^\tau$ \emph{cyclic} when they are semisimple. When $\nu$ is the identity this agrees with Kostant's
definition \cite{Kos} since in that case the root system $\bar R$ is just the standard root system for $\fg^\C$ and
$\alpha_0=-\delta$ where where $\delta$ is the highest positive root relative to the root basis $\{\alpha_1,\ldots,\alpha_r\}$;  this
follows at once from part (1) of Theorem \ref{thm:affine}. It is useful and interesting to understand how the roots in
$\bar R$ are related to a standard root system for $\fg^\C$ when $\nu\neq 1$. This information can be found in 
Appendix \ref{app:roots}
\begin{rem}\label{rem:unique}
It follows from the discussion above that the Coxeter automorphisms are uniquely characterised amongst finite order
automorphisms by the two properties:
\begin{enumerate}
\item $\fg_0^\tau$ is abelian, and
\item $\fg_1^\tau$ admits a non-trivial semisimple element.
\end{enumerate}
\end{rem}

\subsection{Real forms compatible with the Coxeter automorphism.}

Fix the data $(\fg^\C,\nu,\ft^\C,\bar B)$ to obtain a Coxeter automorphism $\tau$.
By \cite[Ch.~X, Thm 5.2]{Hel} there is a $\tau$-invariant compact real form $\fu\subset \fg^\C$. We define 
\[
\ft=\ft^\C\cap \fu=\Span_\R\{ih_0,\ldots,ih_r\}.
\]
As before we denote by $\kappa$ the real involution which fixes $\fu$ pointwise and let $\mu$ be the Hermitian 
inner product \eqref{eq:mu}.
Since $\alpha(h_j)\in\R$ for all roots $\bar\alpha$ it follows that $\kappa:\fg^{\bar\alpha}\to\fg^{-\bar\alpha}$
Now using the same argument as \cite[Ch.~III, Thm 4.2]{Hel} it can be shown that $\fg^{\bar\alpha}$ is orthogonal to
$\fg^{\bar\beta}$ with respect to the Killing form whenever $\bar\alpha+\bar\beta\neq 0$, and
therefore distinct root spaces are orthogonal for $\mu$. Hence the generators $e_j,f_j$ will be $\mu$-orthogonal and
we may choose them so that $f_j=\kappa e_j$. 

Our aim is to classify all the noncompact real forms $\rho$ (equally, all $\sigma=\rho\kappa$) for which $\tau$ preserves 
$\fg$ and $\ft=\fg\cap\ft^\C$ is the $\tau$-fixed
torus in the maximal compact subspace $\fh\subset\fg$. Since $\kappa$ preserves $\ft^\C$ these properties hold
precisely when $\tau$ commutes with $\sigma$ and $\sigma$ fixes $\ft^\C$ pointwise. 
\begin{defn}
We will say a complex involution $\sigma$ is compatible with the Coxeter automorphism $\tau$ for $(\fg^\C,\nu,\ft^\C,\bar
B)$ when $\sigma$ commutes with $\tau$ and fixes $\ft^\C$ pointwise.
\end{defn}
\begin{lem}
A complex involution $\sigma$ is compatible with the Coxeter automorphism $\tau$ for $(\fg^\C,\nu,\ft^\C,\bar B)$ if and only if
it acts by scaling on every root space $\fg^{\bar\alpha}$ and fixes $\ft^\C$ pointwise.
\end{lem}
\begin{proof}
Suppose $\sigma$ is compatible with $\tau$.
Let $\xi\in \fg^{\bar\alpha}$ for $\bar\alpha=(\alpha,j)\in \bar R$. Since $\sigma(h)=h$ for all $h\in\ft^\C$ it follows
that $[h,\sigma(\xi)]=\alpha(h)\sigma(\xi)$ and therefore $\sigma:\fg^{\bar\alpha}\to g^{\bar\beta}$ where $\bar\beta =
(\alpha,j')$. Since $\sigma$ commutes with $\tau$ it preserves the eigenspace $\fg_1^\tau$, and since this is the sum of
root spaces for the simple affine roots $\bar\alpha_0,\ldots,\bar\alpha_r$ (for which $\alpha_0,\ldots,\alpha_r$ are all
distinct) it must preserve each $\fg^{\bar\alpha_j}$. Each of these is one dimensional and hence $\sigma$
acts by scaling (by $\pm 1$). Since every other root space is generated by these roots 
spaces $\sigma$ must act by scaling on every root space.

The converse is clear since $\tau$ also acts by scaling on the root spaces.
\end{proof}
It follows that $\sigma$ is completely determined by its sign $\pm 1$ on the root spaces $\fg^{\bar\alpha}$
and hence $\sigma$ is determined by a map
\begin{equation}\label{eq:ell}
\ell:\bar R \to \{-1,1\}, 
\end{equation}
satisfying
\begin{equation}\label{eq:ell2}
\ell(\bar\alpha+\bar\beta) = \ell(\bar\alpha)\ell(\bar\beta),\ \ell(\bar\alpha)\ell(-\bar\alpha) =1,\
\prod_{j=0}^r\ell(\bar\alpha_j)^{nm_j} = 1.
\end{equation}
The second condition on $\ell$ is equivalent to $\sigma\kappa=\kappa\sigma$.
We could replace the last two conditions by $\ell((0,0))=1$, but they are useful if we choose to describe $\ell$ by its
restriction to $\bar B$.
Clearly at least one sign must be negative on $\bar B$ to give a genuine involution. For example, when $\nu$ is
itself an involution then taking $\sigma=\nu$ corresponds to the map
\[
\ell(\bar\alpha_j) = \begin{cases}
1,\ j=1,\ldots,r\\
-1,\ j=0.
\end{cases}
\]
We will call a root $\bar\alpha$ \emph{compact} when $\ell(\bar\alpha)=1$, since this is equivalent to
$\fg^{\bar\alpha}$ lying in the complexification $\fh^\C$  of the compact summand in the Cartan decomposition of $\fg$. 
When $\ell(\bar\alpha)=-1$ we will say it is \emph{noncompact}.

Given two such maps $\ell,\ell'$ we obtain real forms $G,G'$ of $G^\C$ each containing $T$ in their maximal compact
subgroup. $\tau$-primitive maps into $G/T$ can be identified with those into $G'/T$ when there is an isomorphism $G/T\to G'/T$
identifying primitive distributions. The existence of a real group isomorphism $\phi:G\to G'$ is equivalent to the existence 
of a complex linear automorphism $\chi$ of $\fg^\C$ which maps $\fg$ to $\fg'$: $\chi$ is just the complex linear
extension of the tangent map to $\phi$ at the identity. The additional conditions require that $\chi$ preserves both
$\ft^\C$ and $\fg_1$, and since $\fg_1$ generates $\fg^\C$ this is equivalent to saying $\chi$ 
commutes with $\tau$. 
Now since $\chi$ maps the Cartan decomposition $\fg =\fh+\fm$ to a Cartan decomposition of $\fg'$ we
may, by altering $\chi$ by an inner automorphism if necessary, choose $\chi$ so that 
\[
\chi(\fh) = \fh' = \fg'\cap \fu,\quad \chi(\fm) = \fm'=\fg'\cap \fu.
\]
It follows that $\chi$ commutes with $\kappa$.
Thus the natural equivalence for $\rho,\rho'$ for our situation can be written in terms of $\sigma,\sigma'$, namely, that
there exists $\chi\in\Aut(\fg^\C)$ such that $\sigma'=\chi\sigma\chi^{-1}$ and $\chi$ commutes with $\tau$.
When these conditions holds we will say $\sigma,\sigma'$ are \emph{$\tau$-equivalent}.
This leads us to the following characterisation. 
\begin{prop}\label{prop:classification}
There is a one-to-one correspondence between complex involutions $\sigma$ which are compatible with the Coxeter
automorphism $\tau$ for $(\fg^\C,\nu,\ft^\C,\bar B)$ and additive maps $\ell$ of the type \eqref{eq:ell}. Two
such involutions $\sigma,\sigma'$ are $\tau$-equivalent if and only if $\ell,\ell'$ satisfy $\ell = \ell'\circ w$
where $w:\bar R\to \bar R$ is a symmetry of the affine root system which preserves $\bar B$. All such symmetries come
from symmetries of the affine Dynkin diagram.
\end{prop}
\begin{proof}
The first statement is clear from the discussion above. Now suppose $\sigma' = \chi\sigma\chi^{-1}$ for some
$\chi\in\Aut(\fg^\C)$ which preserves $\ft^\C$ and $\fg_1$. Clearly $\chi:\fg_j^\sigma\to \fg_j^{\sigma'}$ and therefore
whenever $h\in\ft^\C$ and $\xi\in\fg^{\bar\alpha}$ we have 
\[
[h,\xi]=\alpha(h)\xi\ \implies [\chi(h),\chi(\xi)]=(\alpha\circ\chi^{-1})(\chi(h))\chi(\xi),
\]
so that $(\alpha,j)\in\bar R$ implies $(\alpha\circ\chi^{-1},j)\in\bar R$.
Since every automorphism of $\fg^\C$ leaves the Killing form invariant $\chi$ induces a symmetry of $\bar R$:
\[
w:\bar R\to\bar R;\quad (\alpha,j)\mapsto (\alpha\circ\chi^{-1},j),
\]
for which
\begin{equation}\label{eq:w}
\chi:\fg^{\bar\alpha}\to \fg^{w(\bar\alpha)}.
\end{equation}
Now  let $\xi\in\fg^{\bar\alpha}$ be non-zero, then
\begin{equation}\label{eq:ell'}
\ell(\bar\alpha)\xi = \sigma(\xi) = \chi^{-1}\sigma'(\chi\xi) = \ell'(w(\bar\alpha))\xi.
\end{equation}
Thus $\ell=\ell'\circ w$. Clearly $w$ must preserve $\bar B$ since $\chi$ preserves $\fg_1$. 

Conversely, suppose $\ell=\ell'\circ w$. A standard argument (cf.\ \cite[Ch.~IX, Thm 5.1]{Hel}) shows that $w$ extends to an
automorphism $\chi$ of $\fg^\C$ satisfying \eqref{eq:w}. Now rearranging \eqref{eq:ell'} shows that on every root space
$\sigma'=\chi\sigma\chi^{-1}$. Since the root spaces generate $\fg^\C$ as a Lie algebra, this equation holds on all of $\fg^\C$.
\end{proof}
When we combine this construction of compatible involutions with the classification of noncompact symmetric spaces in
\cite{Hel} we deduce that every noncompact symmetric space has at least one compatible Coxeter automorphism.
\begin{cor}\label{cor:compatible}
For every noncompact symmetric space $G/H$ of Type III there is at least one compatible Coxeter automorphism. 
\end{cor}
\begin{proof}
First we recall from \cite[Ch.~X, \S 5]{Hel} that each symmetric space is determined by a involution of type $(s_0,s_1,\ldots,s_r;n)$,
obtained by labelling the $j$-th vertex of the affine Dynkin diagram  $\Gamma^{(n)}$ with $s_j\in\{0,1\}$ such that
$\sum_{j=0}^r nm_js_j=2$. The involution fixes $\ft^\C$ and acts by $(-1)^{s_j}$ on $\fg^{\bar\alpha_j}$. Thus we choose
$\ell$ to have $\ell(\bar\alpha_j) = (-1)^{s_j}$. 
\end{proof}
Finally, we can be quite explicit about what the involution is for a given map $\ell$. First, by \eqref{eq:ell2}
\[
\prod_{j=0}^r\ell(\bar\alpha_j)^{m_j} = \pm 1
\]
and can only equal $-1$ when $n=2$.
\begin{lem}\label{lem:h}
Let $\ell:\bar R\to \{-1,1\}$ satisfy \eqref{eq:ell2} and $\sigma$ be the involution determined by it.
Let $h\in\ft^\C$ be the unique element which satisfies, for $j=1,\ldots,r$,
\[
\alpha_j(h) = \begin{cases} 
0\ \text{when}\ \ell(\bar\alpha_j)=1,\\
1\ \text{when}\ \ell(\bar\alpha_j)=-1.\end{cases}.
\]
If $\prod_{j=0}^r\ell(\bar\alpha_j)^{m_j}=1$ then $\sigma = \exp(i\pi\ad h)$ and is hence an inner automorphism. 
If $\prod_{j=0}^r\ell(\bar\alpha_j)^{m_j}=-1$ then $\sigma = \nu \exp(i\pi\ad h)$ and is therefore an outer automorphism.
\end{lem}
\begin{proof}
	For simplicity set $\ell_j=\ell(\bar\alpha_j)$. For $j=1,\ldots, r$ 
\[
\exp(i\pi\ad h)\cdot e_j = e^{i\pi\alpha_j(h)}e_j = \ell_je_j.
\]
Since $\alpha_0 = -\sum_{j=1}^r m_j\alpha_j$, when $\prod_{j=0}^r\ell_j^{m_j}=1$ we have
\[
\exp(i\pi\ad h)\cdot e_0 = e^{-i\pi\sum_{j=1}^rm_j\alpha_j(h)}e_j = \prod_{j=1}^r\ell_j^{-m_j}e_j.
\]
Hence when $\prod_{j=0}^r\ell_j^{m_j}=1$ we have $\exp(i\pi\ad h)\cdot e_0=\ell_0e_0$ and therefore
$\sigma=\exp(i\pi \ad h)$. 

When $\prod_{j=0}^r\ell_j^{m_j}=-1$ then $\exp(i\pi\ad h)\cdot e_0=-\ell_0e_0$. But in this case $\nu$ is an
involution with $\nu(e_j)=e_j$ for $j=1,\ldots,r$ and $\nu(e_0)=-e_0$, therefore $\sigma=\nu\exp(i\pi\ad h)$.
\end{proof}
An important consequence of this Lemma is that the symmetric space $G/H$ need not be an outer symmetric space when
the Coxeter automorphism is outer. 
\begin{exam}
As an illustration, consider the diagram $\fa_2^{(2)}$ (see Appendix \ref{app:diagrams}). 
The root system is
\[
\bar R = \{(0,0), (0,1), \pm\bar\alpha_0,\pm\bar\alpha_1,\pm(\bar\alpha_0+\bar\alpha_1)\},
\]
where $\bar\alpha_1=(\alpha_1,1)$ and $\bar\alpha_0=(-2\alpha_1,1)$. This diagram has three distinct labellings:
\begin{enumerate}
\item $(\ell_0,\ell_1)=(-1,1)$ with $\ell_0\ell_1^2=-1$. Hence $\sigma$ is an outer automorphism with fixed-point subalgebra
\[
\fh^\C=\fg^{(0,0)}\oplus \fg^{\bar\alpha_1}\oplus\fg^{-\bar\alpha_1}\simeq \fa_1.
\]
\item $(\ell_0,\ell_1)=(1,-1)$ with $\ell_0\ell_1^2=1$. Hence $\sigma$ is an inner automorphism with fixed-point subalgebra
\[
\fh^\C=\fg^{(0,0)}\oplus\fg^{(0,1)}\oplus \fg^{\bar\alpha_0}\oplus\fg^{-\bar\alpha_0}\simeq \fa_1\oplus\C.
\]
\item $(\ell_0,\ell_1)=(-1,-1)$ with $\ell_0\ell_1^2=-1$. Hence $\sigma$ is an outer automorphism with fixed-point subalgebra
\[
\fh^\C=\fg^{(0,0)}\oplus \fg^{\bar\alpha_0+\bar\alpha_1}\oplus\fg^{-\bar\alpha_0-\bar\alpha_1}\simeq \fa_1.
\]
\end{enumerate}
In the second case the real form $G$ is $PU(2,1)$ with inner symmetric space $\CH^2$, the complex hyperbolic plane.
The first and third cases yield isomorphic real forms $G\simeq G'\simeq PSL(3,\R)$, but there is no isomorphism $G/T\to
G'/T$ which identifies the respective primitive distributions $[\fg_1]$. These labellings give three inequivalent versions of the 
Toda equations, sometimes referred to as equations of Tzitzeica type, with very different geometries and solution
existence properties. In the order given above the equations govern the existence of, respectively, elliptic affine spheres in $\R^3$, 
minimal Lagrangian surfaces in $\CH^2$, and hyperbolic affine spheres in $\R^3$ (equally, convex $\RP^2$
structures). These are discussed quite extensively in the survey article \cite{LofMsurv}.
\end{exam}

\subsection{Totally noncompact pairs.}\label{subsec:totally}

We will say that the pair $(\tau,\sigma)$ is totally noncompact when $\fg_1^\tau\subset \fm$, which is equivalent to $\ell(\bar\alpha_j)=-1$
for $j=0,\ldots,r$. Clearly this determines $\ell$ completely and when $n=1,3$ \eqref{eq:ell2} requires $\sum_{j=0}^rm_j$ to
be even. Hence there exists precisely one totally noncompact pair for each affine Dynkin diagram except for
$\fa_{2k}^{(1)}$. The main reason for paying attention to the totally noncompact case is that it is precisely the case for which 
the theory of principal pairs can be applied to establish the existence of solutions to the Toda equations: 
this is done in \S\ref{sec:pairs}. Therefore it is worthwhile 
giving a classification of the symmetric spaces corresponding to totally noncompact pairs here. We represent
the symmetric spaces by their symmetric pairs $(\fg,\fh)$.
\begin{table}
\begin{center}
\begin{tabular}{|c|c|c|c|c|}
\hline
\text{Diagram} & $\fg$ & $\fh$ & $\rank(G/H)$ & \text{$\fg$ split?} \\ \hline
$\fa^{(1)}_{2k-1}$ & $\fsu(k,k)$ & $\fu(k)\oplus\fsu(k)$ & $k$ & \text{No} \\
$\fb^{(1)}_r$ &$\fso(r+1,r)$ & $\fso(r+1)\oplus \fso(r)$ & $r$ & \text{Yes}\\
$\fc^{(1)}_r$&$\fsp(r,\R)$ & $\fu(r)$ & $r$ & \text{Yes} \\
$\fd^{(1)}_{2k}$&$\fso(2k,2k)$ & $\fso(2k)\oplus\fso(2k)$ & $2k$ & \text{Yes} \\
$\fd^{(1)}_{2k+1}$&$\fso(2k+2,2k)$ & $\fso(2k+2)\oplus\fso(2k)$ & $2k$ & \text{No} \\
$\fe^{(1)}_6$&$\fe_{6(2)}$ & $\fsu(6)\oplus\fsu(2)$ & $4$ & \text{No} \\
$\fe^{(1)}_7$&$\fe_{7(7)}$ & $\fsu(8)$ & $7$ & \text{Yes} \\
$\fe^{(1)}_8$&$\fe_{8(8)}$ & $\fso(16)$ & $8$ & \text{Yes} \\
$\ff^{(1)}_4$&$\ff_{4(4)}$ & $\fsp(3)\oplus \fsu(2)$ & $4$ & \text{Yes} \\
$\fg^{(1)}_2$&$\fg_{2(2)}$ & $\fsu(2)\oplus \fsu(2)$ & $2$ & \text{Yes} \\
$\fa^{(2)}_{r}$ & $\fsl(r+1,\R)$ & $\fso(r+1)$ & $r$ & \text{Yes} \\
$\fd^{(2)}_{r}$&$\fso(r,r)$ & $\fso(r)\oplus\fso(r)$ & $r$ & \text{Yes} \\
%$\fd^{(2)}_{2k+1}$&$\fso(2k+1,2k+1)$ & $\fso(2k+1)\oplus\fso(2k+1)$ & $2k+1$ & \text{Yes} \\
$\fe^{(2)}_6$&$\fe_{6(6)}$ & $\fsp(4)$ & $6$ & \text{Yes} \\
$\fd^{(3)}_4$ & $\fso(4,4)$ & $\fso(4)\oplus\fso(4)$ & $4$ & \text{Yes} \\
\hline
\end{tabular}
\end{center}
\caption{Noncompact symmetric spaces for totally noncompact pairs.}\label{table1}
\end{table}
\begin{thm}
The noncompact symmetric spaces corresponding to totally noncompact pairs are those listed in Table
\ref{table1}. 
\end{thm}
The notation in Table \ref{table1} follows Helgason \cite[Ch.~X, Table V]{Hel}. In particular, $\fg_{r(\delta)}$ is the
noncompact real form of type $\fg$, rank $r$ and character $\delta$. This is a split real form (called \emph{normal} in \cite{Hel})
precisely when $r=\delta$. 
\begin{proof}
Since the involution $\sigma$ with $\ell(\bar\alpha_j)=-1$ for all $j=0,\ldots,r$ is unique, up to equivalence, when it exists
it is determined by its fixed point subalgebra $\fh^\C$.
Let $x\in\ft^\C$ be the unique element for which $\alpha_j(x)=1$ for $j=1,\ldots,r$. Then $x$ is the semisimple element in
a principal three dimensional subalgebra. For diagrams of type $1$ this is immediate from \cite[Lemma 5.2]{Kos} and for
the diagrams of type $2$ or $3$ it follows from the fact that, for $j=1,\ldots,r$, $\alpha_j$ is the restriction of a simple root to $\ft^\C$
(see Appendix \ref{app:roots}).  
By Lemma \ref{lem:h} the totally noncompact cases have involution $\sigma=\exp(i\pi\ad x)$ when $\sigma$ is inner (i.e, for
affine diagrams of type $1$ or $3$) and $\sigma=\nu\exp(i\pi\ad x)$ 
when $\sigma$ is outer. It is easy to check that
$\exp(i\pi\ad x)$ is rotation by $\pi$ in the associated principal $\fsu(2)$ subalgebra, whose semisimple element is $ix$. Hence these cases
include all the involutions introduced by Hitchin \cite[Remarks 6.11]{Hit}. Since the real form for Hitchin's involution is split and 
there is exactly one split real form for every simple Lie
algebra, this provides all the split real form cases in Table \ref{table1}. In particular, since $\fd_4^{(3)}$ cannot
produce an outer involution, its involution must be the same as $\fd_4^{(1)}$.

For those cases where $\sigma$ is not Hitchin's involution (i.e., $\fa^{(1)}_{2k-1}$, $\fd^{(1)}_{2k+1}$ and $\fe^{(1)}_6$)
we check $\fh^\C$ by hand by considering the root system generated by the roots with
$\ell(\bar\alpha)=1$.

For $\fa^{(1)}_{2k-1}$ it is easy to see that the roots with $\ell(\bar\alpha)=1$ provide two independent simple root systems, each of
type $\fa_{k-1}$, generated by roots of the form $\alpha_j+\alpha_{j+1}$, the distinct cases being where $j$ is odd or $j$ is even. Since
$\ft^\C$ has rank $2k-1$ we have $\fh^\C\simeq \fa_{k-1}\oplus\fa_{k-1}\oplus\C $.

To deal with $\fd^{(1)}_{2k+1}$ we note first that the roots $\alpha_j$, $j\neq 0$, generate a
$\fd_{2k+1}$ system. The root system for $\fd_r$ can be explicitly represented by
\[
R(\fd_{r}) = \{\alpha\in \Z^{r}:|\alpha|^2=2\},
\]
where the length is the standard Euclidean length in $\R^r$ (see, for example, \cite[Ch.\ X,\ \S 3]{Hel}). Let $\e_1,\ldots,\e_r$ 
be the standard basis for $\R^r$, then the generators for $R(\fd_r)$ are 
\[
\alpha_j = \e_j-\e_{j+1}\ (j=1,\ldots,r-1)\ ,\alpha_r=\e_{r-1}+\e_r,
\]
It follows that for $R(\fd_{2k+1})$ we get two independent subsystems generated by
\[
\{\alpha_{2j-1}+\alpha_{2j}:j=1,\ldots, k\}\cup\{\alpha_{2k-1}+\alpha_{2k+1}\}, \quad
\{\alpha_{2j}+\alpha_{2j+1}:j=1,\ldots, k-1\},
\]
which are therefore of type $\fd_{k+1}$ and $\fd_k$ respectively.
The affine roots
with $\ell(\bar\alpha)=1$ are generated by sums of pairs of simple roots, and therefore for $\fd^{(1)}_{2k+1}$ we have
$\fh^\C\simeq \fd_{k+1}\oplus\fd_k$.

For $\fe^{(1)}_6$ one checks that there are $16$ positive roots with $\ell(\bar\alpha)=1$ and these are generated by 
\begin{eqnarray*}
&\gamma_1=\alpha_3+\alpha_4,\ \gamma_2=\alpha_5+\alpha_6,\ \gamma_3=\alpha_2+\alpha_4,\ 
\gamma_4=\alpha_1+\alpha_3,\ \gamma_5=\alpha_4+\alpha_5,\\
&\gamma_6 = \alpha_2+\alpha_3+\alpha_4+\alpha_5.
\end{eqnarray*}
The first five roots generate a system of type $\fa_5$ and $\gamma_6$ is independent.
Thus $\fh^\C\simeq \fa_5\oplus\fa_1$.
\end{proof}
\begin{rem}
A real form $\fg$ is \emph{quasisplit} when the centralizer of a Cartan subspace of $\fm^\C$ is a Cartan
subalgebra of $\fg^\C$. This happens if and only if $\fm^\C$ contains a regular semisimple element of $\fg^\C$.
All the real forms in Table \ref{table1} are quasisplit because $\sum_{j=0}^r e_j$ is regular semisimple. However, there
are quasisplit real forms which do not appear in this table, namely $\fsu(k+1,k)$ and $\fso(2k+3,2k+1)$. 
\end{rem}

\section{Geometric Toda equations.}
In this section we will derive the geometric Toda equations as given in Definition \ref{defn:geomToda}.
From now on we assume that $\Sigma$ is a compact Riemann 
surface of genus $g$ and equip it with a metric of constant curvature
$2-2g$, so that the  K\"{a}hler form satisfies $\int_\Sigma \omega_\Sigma =2\pi$. 
Motivated by Higgs bundles and more generally principal pairs \cite{Ban,BraGM,GarGM}, we want to view equations \eqref{eq:F1} and
\eqref{eq:F2} as equations on the gauge orbit of an $T$-connection $\nabla$ in a holomorphic $T^\C$-bundle $Q^\C$, given
$\varphi\in H^0(Q^\C(\fg_1)\otimes K_\Sigma)$. The gauge group is $\caG=C^\infty(\Sigma,T^\C)$ and these equations will be our geometric Toda 
equations: they are the equations which
ensure that the reduction of structure group $Q\subset Q^\C$ corresponds to a $\tau$-primitive harmonic map using
Corollary \ref{cor:Q}.  Since $T$ acts unitarily on $\fg^\tau_1$
when it is equipped with the metric $\mu$ from \eqref{eq:mu}, a reduction of structure group $Q\subset Q^\C$
equips the holomorphic bundle $Q^\C(\fg^\tau_1)$ with a Chern connection.
\begin{rem}
Note that any complex automorphism of $\fg^\C$ which fixes $\ft^\C$ pointwise, such as $\tau$ or $\sigma$, induces a 
holomorphic automorphism on $Q^\C(\fg^\C)$ which we will denote by the same name. For every reduction of structure group
$Q\subset Q^\C$ the bundle $Q(\fg^\C)$ also inherits the real involutions $\rho,\kappa$.
What follows also applies when $Q^\C$ is a holomorphic $(T')^\C$-bundle where $T'<  T$. Such a reduction can apply when
considering Toda solutions with additional symmetry. Henceforth we will assume that this is understood.
\end{rem}

Given $(Q^\C,\varphi)$, choose an initial $T$-bundle $Q_0\subset Q^\C$. Let $\|\cdot\|$ denote the norm of this initial
metric and $\nabla$ denote the corresponding Chern connection.
The action of the gauge group $\caG$ on connections is such that, for $s\in\caG$,
$s\cdot\nabla$ is the Chern connection compatible with $s\nabla''s^{-1}$. Specifically this means
\[
s\cdot\nabla = \nabla -(s^{-1}\bar\partial s +\rho(s^{-1}\partial s)).
\]
Here $d=\partial+\bar\partial$ is the type decomposition of the exterior derivative on $\Sigma$. To simplify notation
define $\bar s = \rho(s)^{-1} =\kappa(s)^{-1}$, and then we have
\[
s\cdot\nabla = \nabla - (\bar\partial\log s - \partial\log \bar s).
\]
Note that $\Ad s\cdot \varphi$ is holomorphic with respect to $(s\cdot\nabla)''$. Hence
\[
F^{s\cdot\nabla} = F^\nabla - 2\partial\bar\partial\log(s\bar s),
\]
and note that $2\partial\bar\partial = i*\Delta_\Sigma$. Now observe that
\[
[\Ad s\cdot\varphi\wedge\rho(\Ad s\cdot\varphi)] = [\Ad s\bar s\cdot \varphi\wedge \rho\varphi],
\]
since $T^\C$ is abelian. From now on set $u=s\bar s$. Since $\rho=\kappa$ on $T^\C$ it follows that $u:\Sigma\to \exp(i\ft)$
and $\log(u)$ is well-defined.
Using the root spaces we may decompose $Q^\C(\fg^\tau_1) = \oplus_{j=0}^r Q^\C(\fg^{\bar\alpha_j})$ and
hence write $\varphi = \sum_{j=0}^r \varphi_j$ for $\varphi_j\in H^0(Q^\C(\fg^{\bar\alpha_j})\otimes K_\Sigma)$.  
\begin{lem}
For $0\leq j\leq r$,
\begin{equation}
[\Ad u\cdot\varphi_j\wedge\rho\varphi_j] = ie^{\alpha_j(\log u)}\|\varphi_j\|^2\ell_j\frac{|\alpha_j|^2}{2}h_j\omega_\Sigma,
\end{equation}
where $\ell_j=\ell(\bar\alpha_j)$ determine the real involution $\rho$.
\end{lem}
\begin{proof}
First we note that 
\[
\Ad u\cdot\varphi_j = e^{\alpha_j(\log u)}\varphi_j. 
\]
It is also clear that $[\varphi_j\wedge \rho\varphi_k]=0$ for $k\neq j$. Now let $\tau_j$ be a $\mu$-unitary local section
of $Q^\C(\fg^{\bar\alpha_j})$ and write $\varphi_j  = A_j\tau_j dz$ locally, where $A_j$ is a locally smooth function. 
Then
\[
[\varphi_j\wedge\rho\varphi_j] = [\tau_j,\rho\tau_j]|A_j|^2 dz\wedge d\bar z.
\]
Now $\|\varphi_j\|^2 = |A_j|^2\|dz\|^2$ and $dz\wedge d\bar z = -i\|dz\|^2\omega_\Sigma$ so that
\[
|A_j|^2dz\wedge d\bar z = -i\|\varphi_j\|^2\omega_\Sigma.
\]
Finally, $[\tau_j,\rho\tau_j]$ is a local section of the adjoint bundle $Q^\C(\ft^\C)$ which is trivial since $\ft^\C$ is abelian.
To calculate this we need a unit length vector in each $\fg^{\bar\alpha_j}$. We note that
\[
\mu(e_j,e_j) = -\<e_j,f_j\> = -\tfrac12\<[h_j,e_j],f_j\> =\tfrac12\<h_j,h_j\> = \frac{2}{|\alpha_j|^2}, 
\]
by \eqref{eq:hj}.
Therefore relative to the $\mu$-induced metric on $Q^\C(\fg^\tau_1)\simeq Q\times_T \fg^\tau_1$ we can take $\tau_j$ to be the
equivalence class of 
\[
(q,\frac{|\alpha_j|}{\sqrt{2}}e_j)\in Q^\mu\times \fg^{\bar\alpha_j}.
\]
Hence 
\[
[\tau_j,\rho\tau_j] = \frac{|\alpha_j|^2}{2}[e_j,\rho e_j] = -\ell(\bar\alpha_j)\frac{|\alpha_j|^2}{2}h_j.
\]
\end{proof}
Hence \eqref{eq:F1} holds if and only if
\begin{equation}\label{eq:Todagen}
\Delta_\Sigma\log(u) - \sum_{j=0}^r \ell_j e^{\alpha_j(\log u)}\|\varphi_j\|^2\frac{|\alpha_j|^2}{2}h_j+i*F^\nabla=0.
\end{equation}
Since $u:\Sigma\to \exp(i\ft)$ it has a well-defined logarithm $w=\log(u)$. To make this look a little more familiar we define
functions 
\begin{equation}\label{eq:w_j}
w_j = \alpha_j(\log(u)),
\end{equation} 
and partition the index set $I=\{0,1,\ldots,r\}$ into the union of
\[
I_+ = \{j\in I:\ell_j = 1\},\quad I_-=\{j\in I:\ell_j=-1\}.
\]
Then equation \eqref{eq:Todagen} is equivalent to the system of equations
\begin{equation}\label{eq:Toda1}
\Delta_\Sigma w_j = \sum_{k\in I_+}\|\varphi_k\|^2 \tfrac{|\alpha_k|^2}{2} e^{w_k}\hat C_{jk} - 
\sum_{k\in I_-}\|\varphi_k\|^2 \tfrac{|\alpha_k|^2}{2} e^{w_k}\hat C_{jk} - \alpha_j(i*F^\nabla),\quad j=0,\ldots,r.
\end{equation}
These variables satisfy the relation $\sum_{j=0}^r m_jw_j=0$.
\begin{rem}\label{rem:metric}
Note that after the gauge transformation $s$ the new metric $\|\cdot\|_s^2$ on $Q^\C(\fg^{\bar\alpha_j})$ is given by
$\|\cdot\|_s^2 = e^{w_j}\|\cdot\|^2$.
\end{rem}

We can simplify the term $\alpha_j(*F^\nabla)$ by choosing the initial connection as
follows. First we observe that, since we have chosen $G$ to have trivial centre, the basis $\epsilon_1,\ldots,\epsilon_r\in
i\ft$ dual to the root basis $\alpha_1,\ldots,\alpha_r$ for $\Hom(i\ft,\R)$ generates the kernel $\Gamma$ of the exponential map
\[
\be=\exp(2\pi i \cdot):\ft^\C\to T^\C,
\]
(see, e.g., \cite[Ch VIII]{Ser}). Therefore the root basis freely generates the weight lattice
\[
\hat\Gamma = \{\lambda\in(\ft^\C)^*:\lambda(\Gamma)\subset \Z\},
\]
and this is isomorphic to the character group of $T^\C$ by
\[
\hat\Gamma\to\Hom(T^\C,\Ct); \quad \lambda\mapsto \hat\lambda= \be\circ\lambda\circ\be^{-1}.
\]
Therefore we use the root basis to identify $T^\C$ with $(\Ct)^r$. Thus our holomorphic $T^\C$-bundle $Q^\C$ can be canonically
identified with a product of holomorphic $\Ct$-bundles
\[
Q^\C\simeq Q_1\times_\Sigma\ldots\times_\Sigma Q_r;\quad Q_j = Q^\C\times_{\hat\alpha_j} \Ct.
\]
Now let $Q_j(\C)$ denote the line bundle associated to the $\Ct$-bundle $Q_j$. Then 
\[
Q^\C(\fg_1) = \oplus_{j=0}^r Q_j(\C).
\] 
It follows that
\[
Q_0(\C) \simeq \bigotimes_{j=1}^r Q_j(\C)^{-m_j}.
\]
Set $d_j=\deg(Q_j)$, so $d_0=-\sum_{j=1}^r d_jm_j$. Then in $H^2(\Sigma,\Z)$
\[
[\frac{i}{2\pi}\alpha_j(F^{\nabla})] = [\frac{d_j}{2\pi}\omega_\Sigma], \quad j=0,\ldots,r.
\]
Since we are working with line bundles we may choose the initial metric so that $i\alpha_j(F^{\nabla}) =d_j\omega_\Sigma$, 
and therefore
\[
\alpha_j(i*F^\nabla) =  d_j,\ 0\leq j\leq r.
\]
With this choice of initial metric we finally arrive at the system of equations
\begin{equation}\label{eq:geomToda1}
\Delta_\Sigma w_j = \sum_{k\in I_+}\|\varphi_k\|^2\tfrac{|\alpha_k|^2}{2} e^{w_k}\hat C_{jk} - 
\sum_{k\in I_-}\|\varphi_k\|^2\tfrac{|\alpha_k|^2}{2} e^{w_k}\hat C_{jk} - d_j, \quad j\in I,
\end{equation}
for functions $w_j:\Sigma\to \R$ with $\sum_{j=0}^rm_jw_j=0$. These are the \emph{geometric Toda equations} 
of Definition \ref{defn:geomToda}. 
\begin{rem}\label{rem:constantgauge}
We choose not to absorb the multipliers $|\alpha_k|^2/2$ into either $e^{w_k}$, by a
change of variable, or $\|\varphi_k\|^2$, so as to retain the clear link between the initial data and the equations. 
For simply-laced diagrams (those with no multiple edges) $|\alpha_k|^2/2=1$ for all $k$. 
Note also that the curvature only specifies the initial metric $Q\subset Q^\C$ up to constant scaling on each $Q_j(\C)$,
$j=1,\ldots,r$, but any two such 
metrics are equivalent under the action of the constant gauge transformations $g\in T^\C$, $\varphi\mapsto\Ad g\cdot\varphi$.
With the initial metric fixed, the constant gauge transformations
act on solutions by constant translation, $w\mapsto w+\log(g\kappa(g)^{-1})$.
\end{rem}
As one knows from the earlier studies of Toda systems there is a significant qualitative change in solutions when 
one or more $\varphi_k$ is identically zero.  
\begin{defn}
We will call the system \eqref{eq:geomToda} \emph{cyclic} whenever $\varphi_k$ is not identically zero for all
$k=0,\ldots,r$, and refer to $(Q^\C,\varphi)$ as a cyclic Toda pair. Otherwise we will say it is \emph{non-cyclic}. 
We will say it is \emph{simple non-affine} when $\varphi_k=0$ only for $k=0$.
\end{defn}
An obvious necessary condition for the existence of a cyclic pair is that for $k=0,\ldots,r$ each line bundle $Q_k(\C)\otimes K_\Sigma$ must
admit a non-trivial globally holomorphic section $\varphi_k$ and must therefore have non-negative degree. This immediately leads to 
restrictions on the degrees $d_k$. 
\begin{lem}\label{lem:dineq}
A necessary condition for a $T^\C$-bundle $Q^\C$ to admit a cyclic Toda pair is that the vector of degrees $(d_1,\ldots,d_r)\in\Z^r$
lies in the closed convex polytope given by
\begin{equation}\label{eq:d_k}
2-2g\leq d_k,\quad \sum_{k=1}^r m_kd_k\leq 2g-2,\quad k=1,\ldots,r.
\end{equation}
\end{lem}
\begin{cor}
For $g=0$ there are no cyclic pairs. For $g=1$ there is, up to gauge equivalence, a real one-parameter family of cyclic pairs.
\end{cor}
\begin{proof}
The first statement is immediate from \eqref{eq:d_k}. For $g=1$ we are forced to have $d_k=0$ and $Q_k(\C)\simeq \uC$ for all $k$, 
and thus each $\varphi_k$ is constant, so that we may consider $\varphi\in \fg_1^\tau$. Now it is easy to see that each $\Ad_{T^\C}$-orbit
of a cyclic element of $\fg_1^\tau$ contains exactly one element of the form $c\sum_{k=0}^r e_k$ for some
$c\in\R\setminus\{0\}$.
\end{proof}
%\begin{rem}\label{rem:g=1}
%When $g=1$ we may choose $\|\varphi_k\|^2=2/|\alpha_k|^2$ 
%to obtain the system
%\begin{equation}\label{eq:g=1}
%\Delta_\Sigma w_j = \sum_{k\in I_+} e^{w_k}\hat C_{jk} - 
%\sum_{k\in I_-}e^{w_k}\hat C_{jk}, \quad j\in I.
%\end{equation}
%This is the simplest form of the cyclic Toda equations for a noncompact simple Lie group.
%\end{rem}
Now let us consider the non-cyclic cases. 
\begin{prop}\label{prop:non-cyclic}
Suppose $(Q^\C,\varphi)$ is a non-cyclic Toda pair with $\varphi_k=0$ for $k\not\in J$ where $J\subset I$ is proper and non-empty.
Then \eqref{eq:geomToda} reduces to a union of simple non-affine Toda systems for $\{w_j:j\in J\}$ while for $l\not\in J$, $w_l$ and
$d_l$ are determined by a linear combination of $\{w_j:j\in J\}$ and $\{d_j:j\in J\}$ respectively.  
\end{prop}
\begin{proof}
Set $f_k = \ell(\bar\alpha_k)\|\varphi_k\|^2\tfrac{|\alpha_k|^2}{2} e^{w_k}$. Since $f_k=0$ for $k\not\in J$ we can write \eqref{eq:geomToda} simply as
\begin{equation}\label{eq:gT}
\Delta_\Sigma w_l+d_l = \sum_{k\in J} \hat C_{lk}f_k, \quad l\in I.
\end{equation}
Let $B$ be the submatrix of $\hat C$ obtained by removing the rows and columns not indexed by $J$. This is the (possibly
decomposable) Cartan matrix for the subdiagram of the affine Dynkin diagram for $\hat C$ obtained by removing the vertices
labelled by $I\setminus J$. Therefore 
\[
\Delta_\Sigma w_j+d_j = \sum_{k\in J} B_{jk}f_k,\quad j\in J.
\]
This is a union of simple non-affine Toda systems. Since every Cartan matrix is invertible, the matrix $B$ has inverse $A$
so that
\[
\sum_{j\in J} A_{kj}(\Delta_\Sigma w_j+d_j) = f_k,\quad k\in J.
\]
Substituting this into \eqref{eq:gT} for $l\not\in J$ leads to 
\[
\Delta_\Sigma (w_l - \sum_{j,k\in J} \hat C_{lk} A_{kj} w_j) = \sum_{j,k\in J} \hat C_{lk}A_{kj} d_j -d_l.
\]
Since $\Sigma$ is compact the integral of the left hand side is zero and therefore the right hand side is zero, which
determines each $d_l$. Then, after possibly a constant gauge transformation which leaves $\varphi$ invariant, 
\[
w_l = \sum_{j,k\in J} \hat C_{lk} A_{kj} w_j.
\]
\end{proof}

\subsection{Non-cyclic Toda pairs which are not simple non-affine.}\label{sec:non-affine}
Here is the geometric interpretation of non-cyclic Toda pairs which are not simple non-affine. On the Dynkin diagram
$\Gamma^{(n)}$ label the vertices by $s_j=1$ if $j\in I\setminus J$ (i.e., $\varphi_j= 0$) and $s_j=0$ otherwise.
Using the method of Kac \cite{Kac}
this provides an automorphism $\gamma\in\Aut(\fg^\C)$ of type $(s_0,\ldots,s_r;n)$
which commutes with $\sigma$, $\kappa$ (and hence $\rho$) and $\tau$. 
It has order $o(\gamma)=\sum_{j=0}^rnm_js_j$. Inspection shows that 
$o(\gamma)=1$ if and only if $n=1$ and the Toda pair is equivalent to a simple non-affine Toda pair by a symmetry of the extended
Dynkin diagram $\Gamma^{(1)}$. Otherwise the fixed point subalgebra $\fa=\fg_0^\gamma$ is 
a reductive proper subalgebra of $\fg^\C$ which we can write as a Lie algebra direct sum
\[
\fa = \fa_1\oplus\ldots\oplus\fa_k\oplus\fz.
\]
Here each $\fa_j$ is simple and is determined by a connected subdiagram of the nodes labelled by $J$. The ideal
$\fz$ is the centre and it is contained in $\ft^\C$. Let $A=G^\gamma < G$ have Lie algebra $\fa\cap\fg$, and
similarly $A_l,Z<G$ have Lie algebras $\fa_l\cap\fg$ and $\fz\cap \fg$ respectively. Then
\[
A\simeq A_1\times \ldots\times A_k \times Z.
\]
Moreover, $T< A$ and we can write $T\simeq T_1\times \ldots\times T_k\times Z$ where $T_l<A_l$ is the $\tau$-fixed torus in $A_l$. 
The factor $A_l$ will be noncompact provided $\ell(\bar\alpha_j)=-1$ for some node on the connected subdiagram
corresponding to $\fa_l$.
Clearly, for each $A_l$ we have a simple non-affine Toda system and each determines a primitive harmonic map into
$A_j/T_j$ (using the primitive distribution for non-affine roots). Now
\[
A/T\simeq (A_1/T_1)\times\ldots\times (A_k/T_k),
\]
since $T/Z\simeq T_1\times\ldots\times T_k$. Therefore the product of maps into $A_l/T_l$ corresponds to the map obtained
from the original non-cyclic Toda pair and factors through the (totally geodesic) embedding $A/T\subset G/T$. 

\subsection{Inner versus outer Coxeter automorphisms.}\label{sec:innerouter}
Most of the literature considers only the Toda equations for an extended Dynkin diagram (i.e., the inner Coxeter
automorphism). It was observed by Olive \&
Turok \cite{OliT} that when the Dynkin diagram admits a non-trivial symmetry this can
be imposed upon those Toda equations to obtain a reduced system.
The aim here is to explain how to interpret the geometry of this when $\nu$ is an involution, which is the
only case relevant to a discussion of real forms. 

Suppose the Dynkin diagram of $\fg^\C$ admits a non-trivial involution.
As before let $\nu$ denote both the symmetry on the nodes
of the diagram and the corresponding outer automorphism of $\fg^\C$. Let $\tau'=\nu\tau_\inn$ and recall from Remark
\ref{rem:innerouter} that this is not a Coxeter automorphism unless $\fg^\C\simeq \fa_{2k}$. When $\tau_\inn$ has order
$m$ the order of $\tau'$ is $m$ if $m$ is even and $2m$ otherwise. If $\tau_\inn(\xi)=\lambda\xi$ and $\nu(\xi)=\xi$ then
$\tau'(\xi)=\lambda\xi$ and therefore
\[
\fg_1^{\tau_\inn}\cap \fg_0^\nu \supset \fg^{\tau'}_j,\quad j= 
\begin{cases} 1,\ o(\tau')=m,\\ 2,\ o(\tau')=2m.\end{cases}
\]
For the Toda equations \eqref{eq:geomToda} to admit the symmetry of $\nu$ we require $\ell(\alpha_j) =
\ell(\alpha_{\nu(j)})$ and the Toda pair $(Q^\C,\varphi)$ to admit this symmetry, i.e., there is an isomorphism
from $Q_j(\C)$ to $Q_{\nu(j)}$ which identifies $\varphi_j$ with $\varphi_{\nu(j)}$.
In particular the involution $\sigma$
corresponding to $\ell$ commutes with $\nu$ and $\sigma'=\nu\sigma$ is an outer involution. Let $G<G^\C$ be the
real form determined by $\rho=\sigma\kappa$ and $G'$ the real form determined by $\rho'=\sigma'\kappa$, 
with maximal compact subgroups $H,H'$ and
in each of these maximal toral subgroups $T,T'$, where $T'<T$ is the $\nu$-fixed subgroup. Using $\fg^{\tau'}_j$ above we
can equip $G'/T'$ with a horizontal $F$-structure.

Now set $A=G\cap G'$ with Lie
algebra $\fa=\fg\cap\fg'\subset \fg_0^\nu$ and notice $T'<A$. Choose a $(T')^\C$-bundle $Q^\C$ 
and $\varphi\in H^0(Q^\C(\fg_j^{\tau'})\otimes K)$, then $\nu\varphi=\varphi$.
The Toda equations for this
Toda pair $(Q^\C,\varphi)$ are the equations which govern the existence of a $T'$-subbundle $Q\subset
Q^\C$ so that the connection 
\[
\nabla+\varphi+\rho\varphi =\nabla+\varphi+\rho'\varphi,
\]
is flat. Because of the symmetry this is an $A$-connection and provides an equivariant map into $A/T'$ with holonomy
representation $\chi:\pi_1\Sigma\to A$. When we post-compose with $A/T'\to G/T$ we obtain a primitive harmonic map which
projects onto a harmonic map $f_1:\tilde\Sigma\to G/H$ into the inner symmetric space. On the other hand, post-composition with $A/T'\to
G'/T'$ gives an $F$-holomorphic map which projects down to a harmonic map $f_2:\tilde\Sigma\to G'/H'$ into
the outer symmetric space. Note that these maps $f_1,f_2$ are equivariant with respect to \emph{the same} holonomy
representation $\chi$.  When $\varphi$ is cyclic the map $f_2$ does not agree with any map obtained using
a cyclic Toda pair for an outer Coxeter automorphism except when $\fg^\C\simeq \fa_{2k}$ since the two primitive distributions are inequivalent.

\section{Existence of solutions via principal pairs.}\label{sec:pairs}

To apply the theory of principal pairs \cite{BraGM,GarGM} (cf.\ \cite{Ban})
it is necessary to have $\fg^\tau_1\subset\fm^\C$. This is because the principal pair equations for
this situation are 
\begin{eqnarray*}
 F^\nabla -[\varphi\wedge \kappa\varphi]&=&0,\\
\nabla''\varphi&=&0,
\end{eqnarray*}
and therefore we require $\rho\varphi=-\kappa\varphi$ for the first equation to agree with \eqref{eq:F1}.
Thus we need $\sigma=-1$ on $\fg_1^\tau$ and therefore the pair $(\tau,\sigma)$ is totally noncompact.
Given this, a pair $(Q^\C,\varphi)$
consisting of a holomorphic $T^\C$-bundle $Q^\C$ and a holomorphic section $\varphi$  of $Q^\C(\fg^\tau_1)\otimes K_\Sigma$, is
an example of a $K_\Sigma$-twisted principal pair \cite{GarGM}. We will refer to $(Q^\C,\varphi)$
as a \textit{totally noncompact Toda pair}.

Because the Toda equations \eqref{eq:geomToda} are equivalent to the system \eqref{eq:F1} and \eqref{eq:F2} they possess a solution
when $(Q^\C,\varphi)$ is a $0$-polystable principal pair. Since $T^\C$ is abelian the conditions for $0$-stability are relatively easy
to formulate, following \cite{GarGM}. For each $\chi\in\Hom(i\ft,\R)$ let $h_\chi\in i\ft$ be given by $\chi=\<h_\chi,\cdot\>$. Define
\[
B_\chi^-= \{\xi\in\fg^\tau_1: \Ad\exp(th_\chi)\cdot \xi \text{ is bounded as } t\to\infty,\ t\in\R^+\}
\]
and its subset
\[
B_\chi^0= \{\xi\in\fg^\tau_1: \Ad\exp(th_\chi)\cdot \xi =\xi\ \text{ for all } t\in\R\}.
\]
Note that these are both $T^\C$-invariant and therefore the associated bundles 
\[
Q^\C(B_\chi^0)\subset Q^\C(B_\chi^-)\subset Q^\C(\fg_1^\tau)
\]
are well-defined. We also write $\chi = \sum_{j=1}^r x_j\alpha_j$, where $x_j\in\R$, and define
\[
\deg(Q^\C)_\chi = \sum_{j=1}^r x_jd_j.
\]
\begin{defn}[cf. \cite{GarGM}]
The pair $(Q^\C,\varphi)$ is $0$-polystable if for every non-zero $\chi\in\Hom(i\ft,\R)$ for which $\varphi\in H^0(Q^\C(B_\chi^-)\otimes
K_\Sigma)$ we have $\deg(Q^\C)_\chi\geq 0$, with equality only if $\varphi\in H^0(Q^\C(B_\chi^0)\otimes K_\Sigma)$.
When $\deg(Q^\C)_\chi>0$ we say the pair is $0$-stable. 
\end{defn}
Note that in the language of \cite{GarGM} $\chi$ is an antidominant character. The main theorem we wish to apply can be
stated as follows, based on the exposition in \cite{GarGM}.
\begin{thm}
Let $(Q^\C,\varphi)$ be a totally noncompact Toda pair which is $0$-polystable. Then the Toda equations \eqref{eq:geomToda} 
possess a solution, and there
is a corresponding equivariant harmonic map $f:\caD\to G/T$. This solution is uniquely determined when $(Q^\C,\varphi)$ is
$0$-stable.
\end{thm}
It is easy to check that cyclic Toda pairs are always $0$-stable.
\begin{prop}\label{prop:cyclic}
Every totally noncompact cyclic Toda pair $(Q^\C,\varphi)$ is $0$-stable. 
\end{prop}
\begin{proof}
We will show that $0$-stability is vacuously satisfied by a cyclic pair. Suppose $\chi\in \Hom(i\ft,\R)$, then 
\begin{equation}\label{eq:Adexp}
\Ad\exp(th_\chi)\cdot\varphi = \sum_{k=0}^r\Ad\exp(th_\chi)\cdot\varphi_k = \sum_{k=0}^r \exp(t\<\alpha_k,\chi\>)\varphi_k.
\end{equation}
This is bounded as $t\to\infty$ if and only if
\[
\<\alpha_k,\chi\>\leq 0\ \text{for all}\ k=0,\ldots,r.
\]
Since $\alpha_0=-\sum_{j=1}^r m_j\alpha_j$ and $m_j>0$ this cannot hold when $\chi\neq 0$.
\end{proof}
By Prop.\ \ref{prop:non-cyclic} the non-cyclic case is entirely determined by simple non-affine cases. To understand the
stability conditions for the non-affine case we need the following ideas. 
Recall that $\epsilon_1,\ldots,\epsilon_r\in i\ft$ is the dual basis, $\alpha_j(\epsilon_k) =\delta_{jk}$. The roots
determine the closure $W$ of the positive Weyl chamber and the interior $\INT(W^o)$ of its polar cone:
\[
W = \{h\in i\ft:\alpha_j(h)\geq 0\},\quad \INT(W^o)=\{t\in i\ft:\<t,h\>< 0\ \forall\ h\in W\}.
\]
\begin{prop}\label{prop:non-affine}
A totally noncompact simple non-affine Toda pair $(Q^\C,\varphi)$ is $0$-polystable if and only if it is $0$-stable, which
occurs if and only if the degrees $d_1,\ldots,d_r$ are such that $2-2g\leq d_j$ and $\sum_{j=1}^r d_j\epsilon_j$ lies in $\INT(W^o)$.
Equally,
\begin{equation}\label{eq:Rd}
2-2g\leq d_j,\quad \sum_{k=1}^rR_{jk}d_k<0,
\end{equation}
where $R_{jk}$ are the entries to the inverse of the matrix whose entries are $\<\alpha_j,\alpha_k\>$.
In particular, this requires $g\geq 2$.
\end{prop}
\begin{proof}
From \eqref{eq:Adexp} we see that $\varphi\in H^0(Q^\C(B_\chi^-)\otimes K_\Sigma)$ precisely when 
\begin{equation}\label{eq:chi}
\<\alpha_k,\chi\>\leq 0\ \text{for all}\ k=1,\ldots,r.
\end{equation}
In particular $\Ad\exp h_\chi\cdot\varphi=\varphi$ implies $\chi= 0$, hence $0$-polystable implies $0$-stable.
Now \eqref{eq:chi} is equivalent to $-h_{\chi}\in W$. If we set $\chi = x_j\alpha_j$ and $d = \sum_{j=1}^rd_j\epsilon_j$ then
we require
\[
0>-\sum_{j=1}^r x_jd_j = -\chi(d) = \<-h_\chi,d\>,\ \forall -h_\chi\in W,
\]
and therefore $d\in \INT(W^o)$. Clearly $h\in W$ if and only if $h=\sum_{k=1}^ra_k\epsilon_k$ with $a_k\geq 0$. Hence the
condition above is equivalent to 
\[
\sum_{j=1}^r d_j\<\epsilon_j,\epsilon_k\> <0,\ \forall\ k=1,\ldots,r.
\]
Now $h_{\alpha_k}=\sum_{k=1}^r\<\alpha_l,\alpha_k\>\epsilon_l$ and therefore
\[
\delta_{kj} = \<h_{\alpha_k},\epsilon_j\> = \sum_{l=1}^r\<\alpha_l,\alpha_k\>\<\epsilon_l,\epsilon_j\>.
\]
Thus we obtain \eqref{eq:Rd}.
\end{proof}
\begin{rem}\label{rem:totally}
It is not necessary for $(\tau,\sigma)$ to be totally noncompact for Prop.\ \ref{prop:non-affine} to hold. Clearly all that is required is that the
roots $\bar\alpha_1,\ldots,\bar\alpha_r$ are noncompact ($\ell(\bar\alpha_0)$ can have either sign). 
\end{rem}

\section{Toda pairs and Higgs bundles.}\label{sec:Higgs}

Whenever the geometric Toda equations \eqref{eq:geomToda} admit solutions there is a corresponding primitive harmonic map
$\psi:\tilde\Sigma\to G/T$ from the universal cover $\tilde\Sigma$ of $\Sigma$ by Theorem \ref{thm:struct}. This is
equivariant with respect to some representation $\chi:\pi_1\Sigma\to G$. By 
Theorem \ref{thm:Black} and Proposition \ref{prop:conformal} after homogeneous projection $\pi:G/T\to G/H$ we obtain an equivariant 
harmonic map $f:\tilde\Sigma\to G/H$. This has a corresponding $G$-Higgs bundle, 
i.e., a holomorphic principal $H^\C$ bundle
$P^\C$ over $\Sigma$ equipped with a holomorphic section $\Phi\in H^0(P^\C(\fm^\C)\otimes K_\Sigma)$. 
For practical purposes we
can take the Higgs bundle to be the pair $(E,\Phi)$ where $E=P^\C(\fm^\C)$.

We can describe fairly easily the relationship between the
Toda pair $(Q^\C,\varphi)$ for $\psi$ and the Higgs bundle $(E,\Phi)$ of $f=\pi\circ\psi$. To achieve this we need to compare the
reductive decompositions of $\fg^\C$ provided by $\tau$ and by $\sigma$. Write
\[
\fg^\C = \ft^\C\oplus\fp^\C = \fh^\C \oplus \fm^\C,
\]
for these. Then $\fp^\C = \fk^\C\oplus\fm^\C$ where $\fk^\C=\fp^\C\cap\fh^\C$. In particular, $\Ad_{T^\C}$ preserves both
$\fk^\C$ and $\fm^\C$ so we have a bundle decomposition $Q^\C(\fp^\C)=Q^\C(\fk^\C)\oplus Q^\C(\fm^\C)$ and the
corresponding decomposition on smooth sections
\begin{eqnarray*}
\Gamma(Q^\C(\fp^\C)\otimes K_\Sigma) & = & \Gamma(Q^\C(\fk^\C)\otimes K_\Sigma) \oplus \Gamma(Q^\C(\fm^\C)\otimes K_\Sigma)\\
\varphi & = & \varphi_\fk + \varphi_\fm.
\end{eqnarray*}
Note that since $\Ad_{T^\C}$ preserves root spaces, for a section $\varphi$ of $Q^\C(\fg_1)\otimes K_\Sigma$
written as $\sum_{j\in I}\varphi_j$ we have
\[
\varphi_\fk = \sum_{j\in I_+}\varphi_j,\quad \varphi_\fm = \sum_{j\in I_-}\varphi_j,
\]
when $\sigma$ corresponds to the partition $I=I_+\cup I_-$. 
\begin{prop}\label{prop:Higgs}
Let $(Q^\C,\varphi)$ be a Toda pair which provides a primitive harmonic map $\psi:\tilde\Sigma\to G/T$. Let $\nabla$ be
the connection and $\rho$ the real involution this induces on $Q^\C(\fm^\C)$ by the reduction of
structure group $Q\subset Q^\C$. Then the Higgs
bundle $(E,\Phi)$ for its projection $f=\pi\circ\psi:\tilde\Sigma\to G/H$ is such that
$E\simeq Q^\C(\fm^\C)$ as a smooth bundle but equipped with the holomorphic structure
$\nabla'' +\rho\varphi_\fk$, and $\Phi = \varphi_\fm$.
\end{prop}
Note that the assumption that $(Q^\C,\varphi)$ corresponds to a primitive harmonic map implies that $(E,\Phi)$ is a
polystable $G$-Higgs bundle since it corresponds to the harmonic map $f$.
\begin{proof}
Recalling \S \ref{ss:homog}, the connections and Maurer-Cartan $1$-forms of respectively $G/T$ and $G/H$ come from the two ways of splitting 
the left Maurer-Cartan form of $G$:
\[
\omega_G = \omega_\ft+\omega_\fp = \omega_\fh+\omega_\fm.
\]
Let us write $\beta_1:T(G/T)\to G\times_\ft\fp$ and $\beta_2:T(G/H)\to G\times_\fh\fm$ for the respective Maurer-Cartan
$1$-forms of these spaces. 
By definition $Q=\psi^{-1}G/\pi_1\Sigma$ and $\varphi+\rho\varphi = \psi^*\beta_1$. Define $P=f^{-1}G/\pi_1\Sigma$, 
then $P\simeq Q\times_TH$ and hence $E\simeq Q^\C(\fm^\C)$ as a smooth bundle. 
This bundle carries two connections: a connection $D$ from $\omega_\fh$ and, since $\Ad_T$ preserves the splitting $\fp=\fk+\fm$, the
connection $\nabla$ on $Q^\C(\fp^\C)$ coming from $\omega_\ft$ restricted to $E$. It follows that
\[
D-\nabla = (\psi^*\beta_1)_\fk = (\varphi+\rho\varphi)_\fk.
\]
Therefore the holomorphic structure of $E$ is $D''=\nabla'' +\rho\varphi_{\fk}$.  
Using a local lift $F:U\to G$ of $\psi$, and hence of $f$, it is easy to see that $f^*\beta_2 = (\psi^*\beta_1)_\fm$ and
therefore
\[
\Phi=(f^*\beta_2)'= (\varphi+\rho\varphi)_\fm' = \varphi_\fm.
\]
\end{proof}
\begin{rem}\label{rem:Higgs}
In general we need to know $Q\subset Q^\C$, i.e.\ have solved the geometric Toda equations, to find $\rho\varphi_\fk$ to
obtain the holomorphic structure of $E$. However, when $\varphi_\fk =0$ (equally $\varphi=\varphi_\fm$), which includes 
the totally noncompact case, the holomorphic 
structure of $E$ comes directly from $Q^\C$ without solving the Toda equations, and $\Phi=\varphi$. 
\end{rem}
An important class of $G$-Higgs bundles are those which are invariant under the $\C^\times$ action $e^t\cdot
(E,\Phi)=(E,e^t\Phi)$. These are usually referred to as \emph{variations of Hodge structure} (or simply \emph{Hodge
bundles}).   
\begin{prop}\label{prop:Hodge}
Suppose $(E,\Phi)$ is a $G$-Higgs bundle obtained from a Toda pair $(Q^\C,\varphi)$ in the manner above. 
Then $(E,\Phi)$ is a Hodge bundle if and only if $(Q^\C,\varphi)$ is non-cyclic.
\end{prop}
\begin{proof}
$(E,\Phi)$ is a Hodge bundle when there is a one-parameter family of gauge transformations $g_t:\Sigma\to H^\C$  for which
\[
\Ad g_t\cdot D'' = D'',\quad \Ad g_t\cdot\Phi = e^t\Phi.
\]
Hence $\Ad g_t$ must restrict to a holomorphic section on each $\End(Q_j(\C))\simeq\uC$ for which $\varphi_j\neq 0$. Therefore without
loss of generality we may assume $g_t$ is constant over $\Sigma$ and $g_t=\exp(\gamma_t)$ for $\gamma_t\in\ft^\C$. This must satisfy
\[
\Ad \exp(\gamma_t)\cdot \varphi_\fm = e^t\varphi_\fm.
\]
Let $J=\{j\in I:\varphi_j\neq 0\}$ and write its partition as $J=J_+\cup J_-$. Then for the equations above to hold we require
integers $k_j\in \Z$, $j\in J$, such that
\begin{equation}\label{eq:alphas}
\alpha_j(\gamma_t) = t + 2\pi  ik_j,\quad j\in J_-.
\end{equation}
If $J=I$ (the cyclic case) then these must satisfy the single equation
\[
0=\sum_{j=0}^rm_j\alpha_j(\gamma_t) = 2\pi i\sum_{j=0}^rm_jk_j + t\sum_{j\in I_-}m_j.
\]
This cannot hold for all $t$ since $\sum_{j\in I_-}m_j\neq 0$.
However, if $J\neq I$ then there is no relation between the $\alpha_j$ for $j\in J$ and therefore \eqref{eq:alphas} can 
be solved for any $t$ (e.g., for all $j\in J$ take $k_j=0$ and choose $\gamma_t=t\xi$ where $\alpha_j(\xi)=1$ for $j\in J$). 
\end{proof}
\begin{exam}
We can illustrate Propositions \ref{prop:Higgs}, \ref{prop:Hodge} and Remark \ref{rem:totally} using the example of $\fa_2^{(1)}$. Let $E_{jk}\in
\fgl(3,\C)$ be the matrix whose only non-zero entry is a $1$ in the $j$-th row and $k$-column. Define
\begin{eqnarray*}
&h_1=E_{22}-E_{11},\ h_2=E_{33}-E_{22},\ h_0=E_{11}-E_{33},\\
&e_1 = E_{21},\ e_2 = E_{32},\ e_0 = E_{13},\quad
 f_1 = E_{12},\ f_2= E_{23},\ f_0 = E_{31}.
\end{eqnarray*}
Then $\{h_j,e_j,f_j:j=0,1,2\}$ are the generators satisfying the relations \eqref{eq:Weyl} for the extended Cartan matrix
\[
\hat C = \begin{pmatrix} 2 & -1 & -1\\ -1 & 2 & -1 \\ -1&-1&2\end{pmatrix}.
\]
There is only one admissible labelling
$\ell:\{0,1,2\}\to\{\pm 1\}$, up to the $\Z_3$ symmetry of the affine diagram
(see Table \ref{table:diagrams}), and we choose this so that $\ell_0=-1$, $\ell_1=1$, $\ell_2 =-1$. The corresponding
Cartan decomposition is schematically
\[
\fh^\C = \left\{\begin{pmatrix} * & * & 0\\ * & * & 0\\ 0&0&*\end{pmatrix}\right\},\quad
\fm^\C = \left\{\begin{pmatrix} 0 & 0 & *\\ 0 & 0 & *\\ *&*&0\end{pmatrix}\right\}.
\]
The noncompact real form is $\fsu(2,1)$ and therefore the noncompact symmetric space is $PU(2,1)/U(2)$. 
The corresponding geometric Toda equations are 
\begin{eqnarray}\label{eq:a21Toda}
\Delta_\Sigma w_0 & = & -\|\varphi_1\|^2 e^{-w_0-w_2} -2 \|\varphi_0\|^2e^{w_0} +\|\varphi_2\|^2e^{w_2} - d_0,\\
\Delta_\Sigma w_2 & = & -\|\varphi_1\|^2 e^{-w_0-w_2} + \|\varphi_0\|^2e^{w_0} -2\|\varphi_2\|^2 e^{w_2}- d_2,\notag
\end{eqnarray}
suppressing the redundant equation for $w_1=-w_0-w_2$. 

It will be convenient to think of each $PU(2,1)$-Higgs bundle as a projective equivalence class of rank $3$ Higgs vector bundles,
of the form $(V\oplus\uC,\Phi)$ where $V$ is a holomorphic rank $2$ bundle, and the Higgs field $\Phi$ is a holomorphic section of 
\[
(\Hom(\uC,V)\oplus \Hom(V,\uC))\otimes K\simeq Q^\C(\fm^\C)\otimes K.
\]
Using the representation of $\fa_2$ above and Prop.\ \ref{prop:Higgs} $V$ is isomorphic 
to $Q_0(\C)\oplus Q_2(\C)^{-1}$ equipped with the holomorphic structure
\begin{equation}\label{eq:dbar}
\begin{pmatrix} \bar\partial_0 & \rho\varphi_1\\ 0 & \bar\partial_2^*\end{pmatrix},
\end{equation}
where $\bar\partial_j$ are $\bar\partial$-operators for the holomorphic structures on $Q_j(\C)$ (with $\bar\partial_j^*$
for the dual line bundle). The Higgs field is 
\[
\Phi =\begin{pmatrix} 0&0&\varphi_0\\ 0&0&0\\ 0&\varphi_2&0\end{pmatrix}.
\]
All such Higgs fields have $\Tr(\Phi^2)=0$ and therefore lie in the nilpotent cone of the $PU(2,1)$-Higgs bundle moduli
space. 

The labelling means the pair $(\tau,\sigma)$ is not totally noncompact, so 
unless $\varphi_1=0$ we cannot apply the existence results of \S \ref{sec:pairs}.
In fact we know from \cite[Prop.\ 5.8]{LofM13} that with the restrictions $w_0=w_2$, $Q_j(\C)\otimes K_\Sigma\simeq\uC$ for 
$j=0,2$, and $\|\varphi_0\|=\|\varphi_2\|=1$ (which produces a version of the $\fa_2^{(2)}$ equations for labelling $(\ell_0,\ell_1)=(1,-1)$) 
there is no Toda solution if $\|\varphi_1\|^2$ is too large. But when $\varphi_1=0$ we can apply
Remark \ref{rem:totally} to Prop.\ \ref{prop:non-affine}. The relevant matrix $R=(R_{jk})$ in the inequalities \eqref{eq:Rd}
is easily computed to be
\[
R = \tfrac23\begin{pmatrix} 2 & 1\\ 1 & 2\end{pmatrix}.
\]
With this matrix the inequalities \eqref{eq:Rd} are 
\begin{equation}\label{eq:d0d2}
-2(g-1)\leq d_0,\quad -2(g-1)\leq d_2,\quad 2d_0+d_2 <0,\quad d_0+2d_2<0.
\end{equation}
which are equivalent to the conditions given in \cite{LofM18} for the Hodge bundle with $\varphi_1=0$ (and $\varphi_j\neq 0$, $j=0,2$) to be
a stable Higgs bundle. The corresponding maps $f:\tilde\Sigma\to \CH^2\simeq PU(2,1)/U(2)$ are exactly the superminimal maps which are not
$\pm$-holomorphic \cite[\S 5.2]{LofM18}.
\end{exam}
\begin{rem}\label{rem:cyclotomic}
In \cite{Sim09} Simpson introduced the notion of a $k$-cyclotomic Higgs bundle.  This is a $G^\C$-Higgs bundle $(E,\Phi)$, with $\Phi\in
H^0(\fg^\C\otimes K_\Sigma)$, which is a fixed point of the action of $\Z_k$ as a subgroup of $\C^\times$. 
It is not hard to see that the Higgs bundles in Prop.\ \ref{prop:Higgs} are only $m$-cyclotomic when $\varphi_\fk=0$. 
\end{rem}

\subsection{Baraglia's cyclic Higgs bundles.}

In \cite{Bar} Baraglia pointed out that certain real forms of the Toda equations arise from special cases within 
the family of Higgs bundles constructed by Hitchin in \cite{Hit}. In our language Baraglia's examples correspond
to totally noncompact Toda pairs for an inner Coxeter automorphism when $Q_j(\C)\simeq
K_\Sigma^{-1}$ for $j=1,\ldots,r$. In that case $\varphi_j$ is constant for $j\neq 0$ and $\varphi_0$ is a holomorphic
differential, a section of
\[
Q_0(\C)\otimes K_\Sigma \simeq \otimes_{j=0}^n K_\Sigma^{m_j} = K_\Sigma^{m/n}.
\]
However, the real form in Hitchin's construction is always the
split real form and for Lie algebras of type $\fa_r$, $\fd_r$ and $\fe_6$ this is the real form for an outer Cartan
involution. But Hitchin's construction also implicitly applies $\nu$-symmetry upon the Higgs field, and therefore
even though Baraglia's examples satisfy the inner version of the Toda equations they also provide harmonic maps into
the outer symmetric space of the split real form: Baraglia's explanation of this is essentially the same as the mechanism 
described in \S \ref{sec:innerouter} above. 
The purpose of this section is to show what happens if we drop the insistence that the Coxeter automorphism be inner.
By Table \ref{table2} (or Remark \ref{rem:innerouter}) the outer Coxeter automorphisms provide different Toda pairs from 
Baraglia's except when $\fg^\C\simeq \fa_{2r}$.

Suppose $(\tau,\sigma)$ is a totally noncompact pair.
We begin by adapting Kostant's construction of the principal three dimensional subgroup \cite{Kos} to make it compatible
with the $\Z_m$ grading for either inner or outer Coxeter automorphisms. For any quadruple $(\fg^\C,\nu,\ft^\C,\bar B)$ 
let $x\in i\ft\subset\fg_0^\tau$ be the 
unique element for which $\alpha_j(x)=1$ for $j=1,\ldots,r$. Then there
are positive constants $c_j\in\R$ for which $x=\sum_{j=1}^r c_j h_j$. Define $e\in\fg_1^\tau$ by
\[
e = \sum_{j=1}^r \sqrt{c_j}e_j,
\]
and set $f=\rho e$.  Then the triple $\{x,e,f\}$ satisfies
\[
[x,e]=e,\quad [x,f] = -f,\quad [e,f]=x,
\]
and generates a principal three dimensional subalgebra $\fs^\C\subset \fg_0^\nu$. In fact from Appendix \ref{app:roots} we see that 
$e=\sum_{j=1}^l a_jE_j$, $a_j\neq 0$, with respect to canonical generators for $\fg^\C$ and therefore  $e$ is also principal nilpotent
for $\fg^\C$. Hence $\fs^\C$ is also a principal three dimensional subalgebra for $\fg^\C$. 

Now take
$Q_j(\C)= K_\Sigma^{-1}$ for $j=1,\ldots,r$ and $\varphi = e$. Note that because we normalised the K\"ahler
metric on $\Sigma$ to have constant curvature $2-2g$ the initial metric for this choice of $Q_j(\C)$ 
is just this K\"ahler metric. This simple non-affine Toda pair $(Q^\C,\varphi)$ is
$0$-stable by Prop.\ \ref{prop:non-affine}. 
\begin{lem}
For $g\geq 2$ the solution to \eqref{eq:geomToda} corresponding to this Toda pair is given by $w_j=\log(2g-2)$ for all
$j=1,\ldots,r$. The metric on each $Q_j(\C)\simeq K_\Sigma^{-1}$ is just the metric on $\Sigma$ of constant curvature $-1$. The
corresponding harmonic map into the symmetric space is a totally geodesic embedding of the Poincar\'e disc into $G/H$.
\end{lem}
\begin{proof}
Since $\varphi_k=\sqrt{c_k}e_k$ and $\|e_k\|^2=2/|\alpha_k|^2$ the equations \eqref{eq:geomToda} are
\[
\Delta_\Sigma w_j + \sum_{k=1}^r \hat C_{jk} c_ke^{w_k} = 2g-2,\quad j=1,\ldots,r.
\]
Now 
\[
\sum_{k=1}^r\hat C_{jk}c_k = \sum_{k=1}^r 2\frac{\<\alpha_j,\alpha_k\>}{|\alpha_k|^2}c_k = \alpha_j(x)=1.
\]
Therefore taking $w_j=\log(2g-2)$ for all $j$ is a solution. Therefore the solution metric on each $Q_j(\C)$ is the
K\"ahler metric of $(2g-2)\omega_\Sigma$, which has constant curvature $-1$. 

Now $ix$, $e+f$ and $i(e-f)$ are $\rho$-fixed and generate a Lie subalgebra $\fs\subset\fg$ 
isomorphic to $\fsu(1,1)$. Let $S<G$ be the corresponding subgroup, with $T_1<T$ the $U(1)$ subgroup tangent to
$ix$. Then $S/T_1\simeq SU(1,1)/U(1)$ is isometric to the Poincar\'e disc and the map $\psi:S/T_1\to G/T$ is the primitive
harmonic map determined by the Toda pair above. Its projection $f:S/T_1\to G/H$ is a totally geodesic embedding.
\end{proof}
To obtain the cyclic version assume $g\geq 1$ and take $\varphi = e + \varphi_0e_0$ where $\varphi_0\in H^0(K_\Sigma^{m/n})$
is non-trivial.  The equations then take the form \eqref{eq:KToda}.
The Higgs bundle $(E,\Phi)$ obtained from this Toda pair has $\Phi=\varphi$ by Prop.\ \ref{prop:Higgs}.
When $\tau$ is inner these are Baraglia's cyclic Higgs bundles \cite{Bar}. 
When $\tau$ is outer the real form is split (see Table \ref{table1}) and the Higgs bundle still fits into 
Hitchin's construction \cite{Hit}, because $[e_0,f]=0$ means $e_0$ is a lowest weight vector for the representation of
of $\fs^\C$, but by Table \ref{table2} it is only the root vector for a lowest root of $\fg^\C$ when $\fg^\C\simeq \fa_{2r}$.

For the $\fa_{2k-1}^{(1)}$, $\fd_r^{(1)}$ and $\fe_6^{(1)}$ cases the real form is not split but $(Q,\varphi)$ is
$\nu$-invariant and the construction of \S \ref{sec:innerouter} applies. This means that $(Q^\C,\varphi)$ 
produces two different types of harmonic map: one into
an inner symmetric space $G/H$ and one into an outer symmetric space $G'/H'$. 
As explained in \S \ref{sec:innerouter} these two maps have the same holonomy representation, taking values in $G\cap G'$. 
The maps into $G'/H'$ come
from Hitchin's construction and their holonomy representations are known as Hitchin representations. 
\begin{exam}
To illustrate the differences between inner and outer Toda equations for the totally noncompact Toda pairs above, let us compare the 
equations for the diagrams $\fa_3^{(1)}$ and $\fa_3^{(2)}$ in the case where the former has the $\nu$-symmetry imposed.

For $\fa_3^{(1)}$ with the $\nu$-symmetry implies $w_3=w_1$ and $w_0=-w_2-2w_1$. A computation gives $c_1=3/2$ and
$c_2=2$. Choosing an arbitrary $\varphi_0\in H^0(K_\Sigma^4)$ gives the Toda equations 
\begin{eqnarray}
\Delta_\Sigma w_1 + 3e^{w_1} -2e^{w_2} - \|\varphi_0\|^2e^{-2w_1-w_2} & = & 2g-2, \notag\\
\Delta_\Sigma w_2 - 3e^{w_1} +4e^{w_2} & = & 2g-2.
\end{eqnarray}
In the standard representation of $\fsl(4,\C)$ we can write the Higgs field as
\[
\begin{pmatrix} 0&\sqrt{c_1} &0&0\\ 0&0&\sqrt{c_2}&0\\ 0&0&0&\sqrt{c_1}\\ \varphi_0 &0&0&0\end{pmatrix}.
\]
These equations correspond to primitive harmonic maps into $PU(2,2)/U(1)^3$ which project down onto the 
inner symmetric space $PU(2,2)/P(U(2)\times U(2))$. But the additional $\nu$-symmetry means this also produce a 
map, as in Baraglia's construction \cite{Bar}, into an outer symmetric space isomorphic to $PSL(4,\R)/H$ where $H\simeq
SO(4)/\{\pm I_4\}$. 

For $\fa_3^{(2)}\simeq\fd_3^{(2)}$ we have $w_0=-w_1-w_2$ and $c_1 = 2$, $c_2=3/2$. Choosing an arbitrary $\varphi_0\in H^0(K_\Sigma^3)$ gives
the Toda equations
\begin{eqnarray}
\Delta_\Sigma w_1 + 3e^{w_1} -2e^{w_2} - \|\varphi_0\|^2e^{-2w_1-w_2} & = & 2g-2, \notag\\
\Delta_\Sigma w_2 - 2e^{w_1} +3e^{w_2} & = & 2g-2.
\end{eqnarray}
In the standard representation the Higgs field for this case is
\[
\begin{pmatrix} 0&\sqrt{c_1} &0&0\\ 0&0&\sqrt{c_2}&0\\ \varphi_0&0&0&\sqrt{c_1}\\ 0 &\varphi_0&0&0\end{pmatrix}.
\]
The Toda pair gives a primitive harmonic map into $PSL(4,\R)/T$ where $T<H$ is a maximal torus, and this projects
to the harmonic map into $PSL(4,\R)/H$ given by Hitchin's construction 
using the Higgs field above.
\end{exam}

\appendix

\section{Simple affine roots when $\nu$ is not the identity.}\label{app:roots}

Given $\fg^\C$ fix a Cartan subalgebra $\fc\subset\fg^\C$ with root system $R\subset\fc^*$. Choose simple roots 
$\beta_1,\ldots,\beta_l\in R$ and
canonical generators $\{H_j,E_j,F_j:1,\ldots,l\}$ which satisfy
\[
[E_j,F_k]=-\delta_{jk}H_j,\quad [H_j,E_k]= C_{kj}E_k,\quad [H_j,F_k]=-C_{jk}F_k,
\]
where $C_{jk}$ are the entries of the Cartan matrix for $\fg^\C$. Let $\nu\in\Aut(\fg^\C)$ be non-trivial and coming from a
symmetry of the Dynkin diagram of $\fg^\C$. Specifically
\[
\beta_{\nu(j)}=\nu^*\beta_j,\ \nu(H_j)=H_{\nu(j)},\ \nu(E_j)=E_{\nu(j)},\ \nu(F_j)=F_{\nu(j)},
\]
where we are also using $\nu$ to denote the symmetry of the labelled vertices on the Dynkin diagram. 
Let $\delta\in R$ denote the highest root.

Our aim is to describe for each such $\nu$ the affine roots
$\bar\alpha_j$, $j=0,\ldots,r$, and their root spaces $\fg^{\bar\alpha_j}$. Most of this information comes directly 
from Helgason \cite[Ch.\ X,\S 5]{Hel}
and the rest follows by fairly straightforward calculation from the same source.  
Let $\ft^\C$ denote the $\nu$-fixed subspace of $\fc$
and $\hat\beta$ denote the restriction of any root $\beta$ to $\ft^\C$. 

First we describe the affine roots for $j\geq 1$, which have the form $(\alpha_j,0)$.
For the $\fa$-type and $\fd$-type cases where $\nu$ is an involution, 
if $\nu(j)=j$ then $\bar\alpha_j=\hat\beta_j$ is a simple
affine root, with root space $\fg^{\bar\alpha_j}=\fg^{\beta_j}$. If $j<\nu(j)$ then $\alpha_j=\hat\beta_j$ but
the root space is
\[
\fg^{\bar\alpha_j}=\Span_\C\{E_j+\nu(E_j)\}.
\]

For $\fe_6^{(2)}$ the involution is $\nu(1)=6$, $\nu(3)=5$, $\nu(2)=2$, $\nu(4)=4$.  We set 
\[
\alpha_1 = \beta_1,\ \alpha_2=\beta_3,\ \alpha_3=\beta_4,\ \alpha_4 = \beta_2.
\]
The root spaces are
\[
\fg^{\bar\alpha_1} = \Span_\C\{E_1+\nu(E_1)\},\ \fg^{\bar\alpha_2} = \Span_\C\{E_3+\nu(E_3)\},\
\fg^{\bar\alpha_3}=\fg^{\beta_4},\ \fg^{\bar\alpha_4} = \fg^{\beta_2}.
\]
%Then 
%\[
%\bar\alpha_0= (-2\alpha_1-3\alpha_2-2\alpha_3-\alpha_4,1) = (\hat\beta_0,1), 
%\]
%where
%\[
%\beta_0 =\beta_2+\beta_3+\beta_4-\delta.
%\]

For $\fd_4^{(3)}$ the $3$-fold symmetry is $\nu(1)=3$, $\nu(2)=2$, $\nu(3)=4$, $\nu(4)=1$. We can take
\[
\alpha_1=\beta_1,\quad \alpha_2 = \beta_2,
\]
with corresponding roots spaces
\[
\fg^{\bar\alpha_1} = \C.\{E_1+\nu(E_1)+\nu^2(E_1)\},\quad \fg^{\bar\alpha_2}=\fg^{\beta_2}.
\]
The highest root of $\fd_4$ is $\delta=\beta_1+2\beta_2+\beta_3+\beta_4$.

It remains to describe $\bar\alpha_0$ for each case.
\begin{prop}
For an affine diagram of type $2$ or $3$, write $\bar\alpha_0=(\hat\beta_0,1)$. Then $\beta_0$ is given by the Table \ref{table2}. 

\begin{table}[h]
\begin{center}
\begin{tabular}{|c|c|}
\hline
Affine diagram & $\beta_0$ \\
\hline
$\fa^{(2)}_{2r}$ & $-\delta$ \\
\hline
$\fa^{(2)}_{2r-1}$ & $\beta_{2r-1}-\delta$  \\
\hline
$\fd^{(2)}_{r+1}$ & $-\sum_{j=1}^r \beta_j$\\
\hline
$\fe^{(2)}_6$ & $\beta_2+\beta_3+\beta_4-\delta$\\
\hline
$\fd^{(3)}_4$ & $\beta_2+\beta_4-\delta$\\
\hline
\end{tabular}
\end{center}
\caption{The roots $\beta_0$ for which $\alpha_0$ is the restriction of $\beta_0$ to $\ft^\C$.}\label{table2}
\end{table}
\end{prop}
\begin{proof}
Since $\alpha_0=-\sum_{j=1}^rm_j\alpha_j$ and $\alpha_j=\hat\beta_j=\hat\beta_{\nu(j)}$ for $j\leq \nu(j)$, it suffices to check that 
Table \ref{table2} satisfies 
\[
\hat\beta_0 = -\sum_{j=1}^r m_j\hat\beta_j.
\]
In each case this can be read off the labelled affine diagrams below.
\end{proof}
We finish by describing the corresponding root spaces $\fg^{\bar\alpha_0}\subset\fg_1^\nu$.
For $\fa^{(2)}_{2r}$ clearly $\nu(\beta_0)=\beta_0$ so $\fg^{\bar\alpha_0}=\fg^{\beta_0}=\fg^{-\delta}$. For all the other
cases $\nu(\beta_0)\neq \beta_0$ so that
\[
\fg^{\bar\alpha_0} = \Span_\C\{E_{\beta_0}-\nu(E_{\beta_0})\}
\]
for $n=2$ and
\[
\fg^{\bar\alpha_0} =\Span_\C\{E_{\beta_0} + \omega^2 \nu(E_{\beta_0}) + \omega\nu^2(E_{\beta_0})\}
\]
for $n=3$ with $\omega = e^{2\pi i/3}$.

\newpage
\section{Table of affine diagrams.}\label{app:diagrams}

\begin{figure}[h]\label{table:diagrams}
\begin{tikzpicture}[scale=0.9]

%LEFT SIDE
\draw(-0.5,0) node {$\fe_8^{(1)}$};
\draw(1,0) circle (5pt);
\draw(2,0) circle (5pt);
\draw(3,0) circle (5pt);
\draw(4,0) circle (5pt);
\draw(5,0) circle (5pt);
\draw(6,0) circle (5pt);
\draw(7,0) circle (5pt);
\draw(8,0) circle (5pt);
\draw(6,1) circle (5pt);
\draw(1.2,0)--(1.8,0);
\draw(2.2,0)--(2.8,0);
\draw(3.2,0)--(3.8,0);
\draw(4.2,0)--(4.8,0);
\draw(5.2,0)--(5.8,0);
\draw(6.2,0)--(6.8,0);
\draw(7.2,0)--(7.8,0);
\draw(6,0.20)--(6,0.80);
\draw(1,-0.4) node {\footnotesize{$\alpha_0$}};
\draw(2,-0.4) node {\footnotesize{$\alpha_8$}};
\draw(3,-0.4) node {\footnotesize{$\alpha_7$}};
\draw(4,-0.4) node {\footnotesize{$\alpha_6$}};
\draw(5,-0.4) node {\footnotesize{$\alpha_5$}};
\draw(6,-0.4) node {\footnotesize{$\alpha_4$}};
\draw(7,-0.4) node {\footnotesize{$\alpha_3$}};
\draw(8,-0.4) node {\footnotesize{$\alpha_1$}};
\draw(6.5,1) node {\footnotesize{$\alpha_2$}};

\draw(1,0) node {\tiny{$1$}};
\draw(2,0) node {\tiny{$2$}};
\draw(3,0) node {\tiny{$3$}};
\draw(4,0) node {\tiny{$4$}};
\draw(5,0) node {\tiny{$5$}};
\draw(6,0) node {\tiny{$6$}};
\draw(7,0) node {\tiny{$4$}};
\draw(8,0) node {\tiny{$2$}};
\draw(6,1) node {\tiny{$3$}};

\draw(-0.5,2) node {$\fe_7^{(1)}$};
\draw(1,2) circle (5pt);
\draw(2,2) circle (5pt);
\draw(3,2) circle (5pt);
\draw(4,2) circle (5pt);
\draw(5,2) circle (5pt);
\draw(6,2) circle (5pt);
\draw(7,2) circle (5pt);
\draw(4,3) circle (5pt);
\draw(1,1.6) node {\footnotesize{$\alpha_7$}};
\draw(2,1.6) node {\footnotesize{$\alpha_6$}};
\draw(3,1.6) node {\footnotesize{$\alpha_5$}};
\draw(4,1.6) node {\footnotesize{$\alpha_4$}};
\draw(5,1.6) node {\footnotesize{$\alpha_3$}};
\draw(6,1.6) node {\footnotesize{$\alpha_1$}};
\draw(7,1.6) node {\footnotesize{$\alpha_0$}};
\draw(4.5,3) node {\footnotesize{$\alpha_2$}};
\draw(1.2,2)--(1.8,2);
\draw(2.2,2)--(2.8,2);
\draw(3.2,2)--(3.8,2);
\draw(4.2,2)--(4.8,2);
\draw(5.2,2)--(5.8,2);
\draw(6.2,2)--(6.8,2);
\draw(4,2.2)--(4,2.8);
\draw(1,2) node {\tiny{$1$}};
\draw(2,2) node {\tiny{$2$}};
\draw(3,2) node {\tiny{$3$}};
\draw(4,2) node {\tiny{$4$}};
\draw(5,2) node {\tiny{$3$}};
\draw(6,2) node {\tiny{$2$}};
\draw(7,2) node {\tiny{$1$}};
\draw(4,3) node {\tiny{$2$}};

\draw(-0.5,4) node {$\fe_6^{(1)}$};
\draw(1,4) circle (5pt);
\draw(2,4) circle (5pt);
\draw(3,4) circle (5pt);
\draw(4,4) circle (5pt);
\draw(5,4) circle (5pt);
\draw(3,5) circle (5pt);
\draw(3,6) circle (5pt);
\draw(1,3.6) node {\footnotesize{$\alpha_6$}};
\draw(2,3.6) node {\footnotesize{$\alpha_5$}};
\draw(3,3.6) node {\footnotesize{$\alpha_4$}};
\draw(4,3.6) node {\footnotesize{$\alpha_3$}};
\draw(5,3.6) node {\footnotesize{$\alpha_1$}};
\draw(3.5,5) node {\footnotesize{$\alpha_2$}};
\draw(3.5,6) node {\footnotesize{$\alpha_0$}};
\draw(1.2,4)--(1.8,4);
\draw(2.2,4)--(2.8,4);
\draw(3.2,4)--(3.8,4);
\draw(4.2,4)--(4.8,4);
\draw(3,4.2)--(3,4.8);
\draw(3,5.2)--(3,5.8);
\draw(1,4) node {\tiny{$1$}};
\draw(2,4) node {\tiny{$2$}};
\draw(3,4) node {\tiny{$3$}};
\draw(4,4) node {\tiny{$2$}};
\draw(5,4) node {\tiny{$1$}};
\draw(3,5) node {\tiny{$2$}};
\draw(3,6) node {\tiny{$1$}};

\draw(-0.5,8) node {$\fd_r^{(1)}$};
\draw(-0.5,7.5) node {\scriptsize{$r>3$}};
\draw(1,8) circle (5pt);
\draw(2,8) circle (5pt);
\draw(3,8) node {$\ldots$};
\draw(4,8) circle (5pt);
\draw(5,8) circle (5pt);
\draw(1,7) circle (5pt);
\draw(1,9) circle (5pt);
\draw(5,7) circle (5pt);
\draw(5,9) circle (5pt);
\draw(1,8) node {\tiny{$2$}};
\draw(1,7) node {\tiny{$1$}};
\draw(1,9) node {\tiny{$1$}};
\draw(2,8) node {\tiny{$2$}};
\draw(4,8) node {\tiny{$2$}};
\draw(5,8) node {\tiny{$2$}};
\draw(5,7) node {\tiny{$1$}};
\draw(5,9) node {\tiny{$1$}};
\draw(0.5,8) node {\footnotesize{$\alpha_2$}};
\draw(0.5,7) node {\footnotesize{$\alpha_1$}};
\draw(0.5,9) node {\footnotesize{$\alpha_0$}};
%\draw(2,7.6) node {\tiny{$2$}};
%\draw(4,7.6) node {\tiny{$2$}};
\draw(5.7,8) node {\footnotesize{$\alpha_{r-2}$}};
\draw(5.7,7) node {\footnotesize{$\alpha_{r-1}$}};
\draw(5.5,9) node {\footnotesize{$\alpha_r$}};
\draw(1.2,8)--(1.8,8);
\draw(1,8.2)--(1,8.8);
\draw(1,7.2)--(1,7.8);
\draw(2.2,8)--(2.6,8);
\draw(3.4,8)--(3.8,8);
\draw(4.2,8)--(4.8,8);
\draw(3,4.2)--(3,4.8);
\draw(3,5.2)--(3,5.8);
\draw(5,8.2)--(5,8.8);
\draw(5,7.2)--(5,7.8);

\draw(-0.5,11) node {$\fc_r^{(1)}$};
\draw(-0.5,10.5) node {\scriptsize{$r>1$}};
\draw(1,11) circle (5pt);
\draw(2,11) circle (5pt);
\draw(3,11) node {$\ldots$};
\draw(4,11) circle (5pt);
\draw(5,11) circle (5pt);
\draw[double](1.20,11)--(1.8,11); \draw(1.7,10.90)--(1.8,11); \draw(1.7,11.1)--(1.8,11);
\draw(2.2,11)--(2.6,11);
\draw(3.4,11)--(3.8,11);
\draw[double](4.2,11)--(4.8,11); \draw(4.3,11.1)--(4.2,11); \draw(4.3,10.9)--(4.2,11);
\draw(1,10.6) node {\footnotesize{$\alpha_0$}};
\draw(2,10.6) node {\footnotesize{$\alpha_1$}};
\draw(4,10.6) node {\footnotesize{$\alpha_{r-1}$}};
\draw(5,10.6) node {\footnotesize{$\alpha_r$}};
\draw(1,11) node {\tiny{$1$}};
\draw(2,11) node {\tiny{$2$}};
\draw(4,11) node {\tiny{$2$}};
\draw(5,11) node {\tiny{$1$}};

\draw(-0.5,13) node {$\fb_r^{(1)}$};
\draw(-0.5,12.5) node {\scriptsize{$r>2$}};
\draw(1,13) circle (5pt);
\draw(2,13) circle (5pt);
\draw(3,13) node {$\ldots$};
\draw(4,13) circle (5pt);
\draw(5,13) circle (5pt);
\draw(1,12) circle (5pt);
\draw(1,14) circle (5pt);
\draw(1,13) node {\tiny{$2$}};
\draw(1,12) node {\tiny{$1$}};
\draw(1,14) node {\tiny{$1$}};
\draw(2,13) node {\tiny{$2$}};
\draw(4,13) node {\tiny{$2$}};
\draw(5,13) node {\tiny{$2$}};
\draw(1.2,13)--(1.8,13);
\draw(1,13.2)--(1,13.8);
\draw(1,12.2)--(1,12.8);
\draw(2.2,13)--(2.6,13);
\draw(3.4,13)--(3.8,13);
\draw[double](4.2,13)--(4.8,13); \draw(4.7,13.1)--(4.8,13); \draw(4.7,12.9)--(4.8,13);
\draw(0.5,13) node {\footnotesize{$\alpha_2$}};
\draw(0.5,12) node {\footnotesize{$\alpha_1$}};
\draw(0.5,14) node {\footnotesize{$\alpha_0$}};
%\draw(2,12) node {\footnotesize{$\alpha_1$}};
\draw(4,12.6) node {\footnotesize{$\alpha_{r-1}$}};
\draw(5,12.6) node {\footnotesize{$\alpha_r$}};

\draw(-0.5,15) node {$\fa_r^{(1)}$};
\draw(-0.5,14.5) node {\scriptsize{$r>1$}};
\draw(1,15) circle (5pt);
\draw(2,15) circle (5pt);
\draw(3,15) node {$\ldots$};
\draw(4,15) circle (5pt);
\draw(5,15) circle (5pt);
\draw(3,16) circle (5pt);
\draw(1.2,15)--(1.8,15);
\draw(2.2,15)--(2.6,15);
\draw(3.4,15)--(3.8,15);
\draw(4.2,15)--(4.8,15);
\draw(1.2,15.1)--(2.8,16);
\draw(3.2,16)--(4.8,15.1);
\draw(1,15) node {\tiny{$1$}};
\draw(3,16) node {\tiny{$1$}};
\draw(2,15) node {\tiny{$1$}};
\draw(4,15) node {\tiny{$1$}};
\draw(5,15) node {\tiny{$1$}};
\draw(1,14.6) node {\footnotesize{$\alpha_1$}};
\draw(2,14.6) node {\footnotesize{$\alpha_2$}};
\draw(3,16.4) node {\footnotesize{$\alpha_0$}};
\draw(4,14.6) node {\footnotesize{$\alpha_{r-1}$}};
\draw(5,14.6) node {\footnotesize{$\alpha_r$}};

\draw(-0.5,17) node {$\fa_1^{(1)}$};
\draw(1,17) circle (5pt);
\draw(2,17) circle (5pt);
\draw[double] (1.2,17.05) -- (1.8,17.05);
\draw[double] (1.2,16.95) -- (1.8,16.95);
\draw(1,17) node {\tiny{$1$}};
\draw(2,17) node {\tiny{$1$}};
\draw(1,16.6) node {\footnotesize{$\alpha_0$}};
\draw(2,16.6) node {\footnotesize{$\alpha_1$}};

%RIGHT SIDE
\draw(9.5,17) node {$\ff_4^{(1)}$};
\draw(11,17) circle (5pt);
\draw(12,17) circle (5pt);
\draw(13,17) circle (5pt);
\draw(14,17) circle (5pt);
\draw(15,17) circle (5pt);
\draw(11.2,17)--(11.8,17);
\draw(12.2,17)--(12.8,17);
\draw[double](13.2,17)--(13.8,17); \draw(13.7,17.1)--(13.8,17); \draw(13.7,16.9)--(13.8,17);
\draw(14.2,17)--(14.8,17);
\draw(11,17) node {\tiny{$1$}};
\draw(12,17) node {\tiny{$2$}};
\draw(13,17) node {\tiny{$3$}};
\draw(14,17) node {\tiny{$4$}};
\draw(15,17) node {\tiny{$2$}};
\draw(11,16.6) node {\footnotesize{$\alpha_0$}};
\draw(12,16.6) node {\footnotesize{$\alpha_1$}};
\draw(13,16.6) node {\footnotesize{$\alpha_2$}};
\draw(14,16.6) node {\footnotesize{$\alpha_3$}};
\draw(15,16.6) node {\footnotesize{$\alpha_4$}};

\draw(9.5,15) node {$\fg_2^{(1)}$};
\draw(11,15) circle (5pt);
\draw(12,15) circle (5pt);
\draw(13,15) circle (5pt);
\draw(11.2,15)--(11.8,15);
\draw(12.2,15)--(12.8,15); \draw(12.2,15.07)--(12.75,15.07); \draw(12.2,14.93)--(12.75,14.93); %triple bond
\draw(12.7,15.15)--(12.8,15); \draw(12.7,14.85)--(12.8,15);
\draw(11,15) node {\tiny{$1$}};
\draw(12,15) node {\tiny{$2$}};
\draw(13,15) node {\tiny{$3$}};
\draw(11,14.6) node {\footnotesize{$\alpha_0$}};
\draw(12,14.6) node {\footnotesize{$\alpha_2$}};
\draw(13,14.6) node {\footnotesize{$\alpha_1$}};

\draw(10,14) -- (15,14);

\draw(9.5,13) node {$\fa_2^{(2)}$};
\draw(11,13) circle (5pt);
\draw(12,13) circle (5pt);
\draw[double] (11.2,13.05) -- (11.75,13.05);
\draw[double] (11.2,12.95) -- (11.75,12.95);
\draw(11.7,13.15)--(11.8,13); \draw(11.7,12.85)--(11.8,13);
\draw(11,13) node {\tiny{$1$}};
\draw(12,13) node {\tiny{$2$}};
\draw(11,12.6) node {\footnotesize{$\alpha_0$}};
\draw(12,12.6) node {\footnotesize{$\alpha_1$}};

\draw(9.5,11) node {$\fa_{2r}^{(2)}$};
\draw(9.5,10.5) node {\scriptsize{$r>1$}};
\draw(11,11) circle (5pt);
\draw(12,11) circle (5pt);
\draw(13,11) node {$\ldots$};
\draw(14,11) circle (5pt);
\draw(15,11) circle (5pt);
\draw[double](11.20,11)--(11.8,11); \draw(11.7,10.90)--(11.8,11); \draw(11.7,11.1)--(11.8,11);
\draw(12.2,11)--(12.6,11);
\draw(13.4,11)--(13.8,11);
\draw[double](14.2,11)--(14.8,11); \draw(14.7,11.1)--(14.8,11); \draw(14.7,10.9)--(14.8,11);
\draw(11,10.6) node {\footnotesize{$\alpha_0$}};
\draw(12,10.6) node {\footnotesize{$\alpha_1$}};
\draw(14,10.6) node {\footnotesize{$\alpha_{r-1}$}};
\draw(15,10.6) node {\footnotesize{$\alpha_r$}};
\draw(11,11) node {\tiny{$1$}};
\draw(12,11) node {\tiny{$2$}};
\draw(14,11) node {\tiny{$2$}};
\draw(15,11) node {\tiny{$2$}};

\draw(9.5,8) node {$\fa_{2r-1}^{(2)}$};
\draw(9.5,7.5) node {\scriptsize{$r>2$}};
\draw(11,8) circle (5pt);
\draw(12,8) circle (5pt);
\draw(13,8) node {$\ldots$};
\draw(14,8) circle (5pt);
\draw(15,8) circle (5pt);
\draw(11,7) circle (5pt);
\draw(11,9) circle (5pt);
\draw(11,8) node {\tiny{$2$}};
\draw(11,7) node {\tiny{$1$}};
\draw(11,9) node {\tiny{$1$}};
\draw(12,8) node {\tiny{$2$}};
\draw(14,8) node {\tiny{$2$}};
\draw(15,8) node {\tiny{$1$}};
\draw(11.2,8)--(11.8,8);
\draw(11,8.2)--(11,8.8);
\draw(11,7.2)--(11,7.8);
\draw(12.2,8)--(12.6,8);
\draw(13.4,8)--(13.8,8);
\draw[double](14.2,8)--(14.8,8); \draw(14.3,8.1)--(14.2,8); \draw(14.3,7.9)--(14.2,8);
\draw(10.5,8) node {\footnotesize{$\alpha_2$}};
\draw(10.5,7) node {\footnotesize{$\alpha_1$}};
\draw(10.5,9) node {\footnotesize{$\alpha_0$}};
\draw(14,7.6) node {\footnotesize{$\alpha_{r-1}$}};
\draw(15,7.6) node {\footnotesize{$\alpha_r$}};

\draw(9.5, 6) node {$\fd_{r+1}^{(2)}$};
\draw(9.5,5.5) node {\scriptsize{$r>1$}};
\draw(11, 6) circle (5pt);
\draw(12, 6) circle (5pt);
\draw(13, 6) node {$\ldots$};
\draw(14, 6) circle (5pt);
\draw(15, 6) circle (5pt);
\draw[double](11.20, 6)--(11.8, 6); \draw(11.3, 5.90)--(11.2, 6); \draw(11.3, 6.1)--(11.2, 6);
\draw(12.2, 6)--(12.6, 6);
\draw(13.4, 6)--(13.8, 6);
\draw[double](14.2, 6)--(14.8, 6); \draw(14.7, 6.1)--(14.8, 6); \draw(14.7, 5.9)--(14.8, 6);
\draw(11, 5.6) node {\footnotesize{$\alpha_0$}};
\draw(12, 5.6) node {\footnotesize{$\alpha_1$}};
\draw(14, 5.6) node {\footnotesize{$\alpha_{r-1}$}};
\draw(15, 5.6) node {\footnotesize{$\alpha_r$}};
\draw(11, 6) node {\tiny{$1$}};
\draw(12, 6) node {\tiny{$1$}};
\draw(14, 6) node {\tiny{$1$}};
\draw(15, 6) node {\tiny{$1$}};

\draw(9.5, 4) node {$\fe_{6}^{(2)}$};
\draw(11, 4) circle (5pt);
\draw(12, 4) circle (5pt);
\draw(13, 4) circle (5pt);
\draw(14, 4) circle (5pt);
\draw(15, 4) circle (5pt);
\draw(11.2, 4)--(11.8, 4);
\draw[double](12.2, 4)--(12.8, 4); \draw(12.7, 4.1)--(12.8, 4); \draw(12.7, 3.9)--(12.8, 4);
\draw(13.2, 4)--(13.8, 4);
\draw(14.2, 4)--(14.8, 4);
\draw(11, 3.6) node {\footnotesize{$\alpha_4$}};
\draw(12, 3.6) node {\footnotesize{$\alpha_3$}};
\draw(13, 3.6) node {\footnotesize{$\alpha_2$}};
\draw(14, 3.6) node {\footnotesize{$\alpha_1$}};
\draw(15, 3.6) node {\footnotesize{$\alpha_0$}};
\draw(11, 4) node {\tiny{$1$}};
\draw(12, 4) node {\tiny{$2$}};
\draw(13, 4) node {\tiny{$3$}};
\draw(14, 4) node {\tiny{$2$}};
\draw(15, 4) node {\tiny{$1$}};

\draw(10,3) -- (15,3);

\draw(9.5, 2) node {$\fd_4^{(3)}$};
\draw(11, 2) circle (5pt);
\draw(12, 2) circle (5pt);
\draw(13, 2) circle (5pt);
\draw(12.2, 2)--(12.8, 2);
\draw(11.2, 2)--(11.8, 2); \draw(11.2, 2.07)--(11.75, 2.07); \draw(11.2, 1.93)--(11.75, 1.93); %triple bond
\draw(11.7, 2.15)--(11.8, 2); \draw(11.7, 1.85)--(11.8, 2);
\draw(11, 2) node {\tiny{$1$}};
\draw(12, 2) node {\tiny{$2$}};
\draw(13, 2) node {\tiny{$1$}};
\draw(11, 1.6) node {\footnotesize{$\alpha_0$}};
\draw(12, 1.6) node {\footnotesize{$\alpha_1$}};
\draw(13, 1.6) node {\footnotesize{$\alpha_2$}};

\end{tikzpicture}
\caption{Affine Dynkin diagrams. There are $r+1$ vertices for all diagrams for which $r$ is a parameter. The numbers in the
circles are the coefficients $m_j$ for which $\sum_{j=0}^rm_j\alpha_j=0$, $m_0=1$.}
\end{figure}
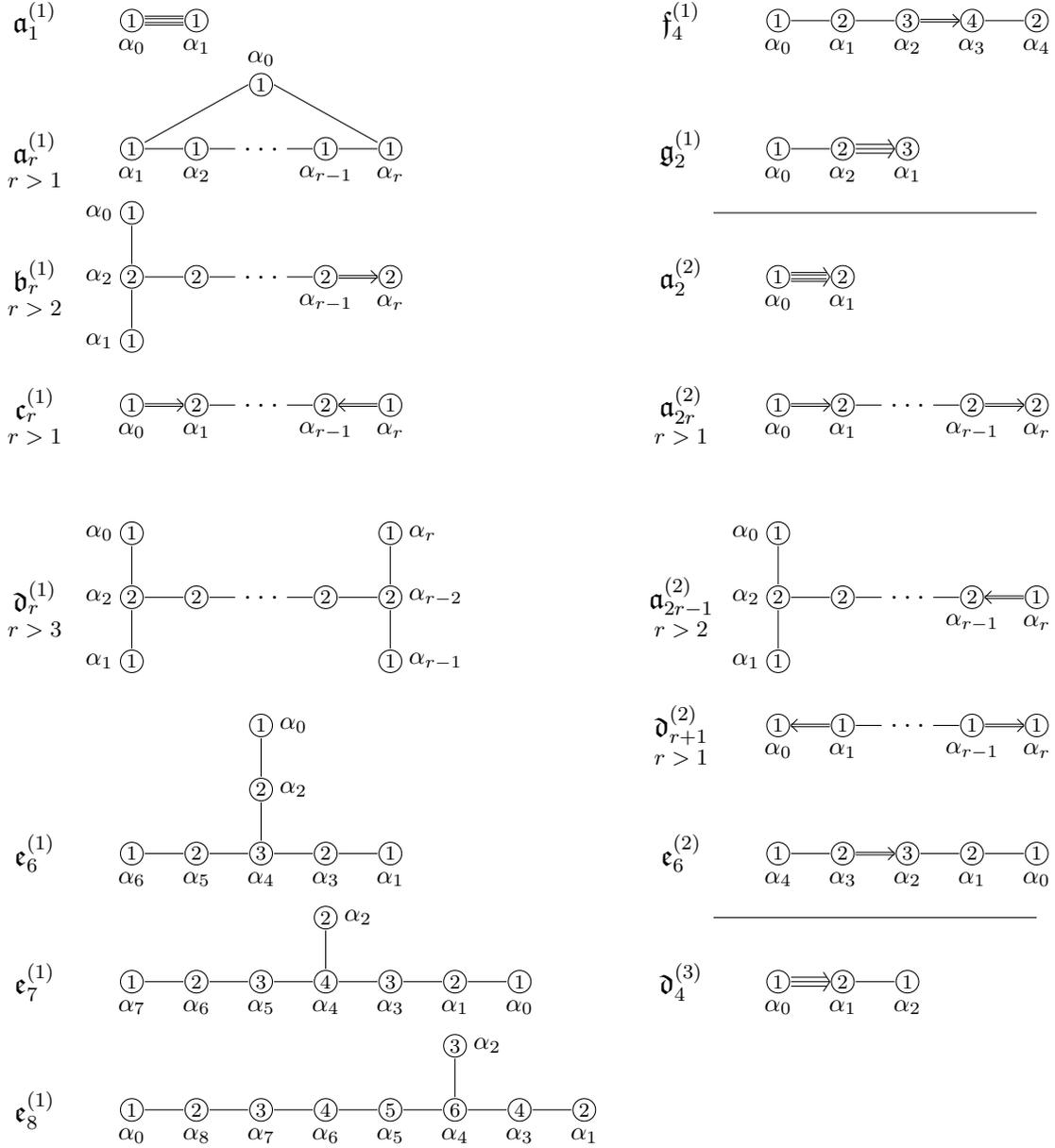

\end{document}